\documentclass{article}

\usepackage{fullpage}
\usepackage{amsthm}
\usepackage{lipsum}
\usepackage{amsfonts}
\usepackage{graphicx,subfigure}
\usepackage{epstopdf}
\usepackage{algorithmic}
\usepackage{xr-hyper}

\usepackage{nicefrac}
\usepackage{bm}
\usepackage{bbm}
\usepackage{amssymb}
\usepackage{amsmath}
\usepackage{mathabx}
\usepackage{mathrsfs}
\usepackage{mathtools}
\usepackage{empheq}
\usepackage{color}
\usepackage[normalem]{ulem}
\usepackage{stmaryrd}
\usepackage{multirow}
\usepackage[colorinlistoftodos,prependcaption,textsize=tiny]{todonotes}
\usepackage{extarrows}
\usepackage[]{hyperref}
\usepackage{cleveref}
\usepackage{xspace}

\ifpdf
  \DeclareGraphicsExtensions{.eps,.pdf,.png,.jpg}
\else
  \DeclareGraphicsExtensions{.eps}
\fi

\def\softd{{\leavevmode\setbox1=\hbox{d}%
          \hbox to 1.05\wd1{d\kern-0.4ex{\char039}\hss}}}

\newcommand{\PO}[1]{{\color{black}{#1}}}

\newcommand{\dx}{\Delta x}
\newcommand{\dt}{\Delta t}
\newcommand{\bbu}{\mathbf{u}}
\newcommand{\bbU}{\mathbf{U}}

\newcommand{\bbF}{\mathbf{F}}
\newcommand{\bbf}{\mathbf{f}}

\newcommand{\uh}{u_{\mathrm{h}}}

\newcommand{\R}{\mathbb{R}}
\newcommand{\jph}{{j+\frac{1}{2}}}
\newcommand{\jmh}{{j-\frac{1}{2}}}

\newcommand*\xbar[1]{%
  \hbox{%
    \vbox{%
      \hrule height 0.5pt 
      \kern0.4ex
      \hbox{%
        \kern-0.05em
        \ensuremath{#1}%
        \kern-0.00em
      }%
    }%
  }%
}

\newtheorem{thm}{Theorem}[section]

\newtheorem{lemma}[thm]{Lemma}

\newtheorem{proposition}[thm]{Proposition}

\newtheorem{rmk}[thm]{Remark}
\newcommand{\dpar}[2]{\dfrac{\partial #1}{\partial #2}}
\newcommand{\Z}{\mathbb Z}
\newcommand{\N}{\mathbb N}

\renewcommand{\P}{\mathbb P}

\newcommand{\bba}{\mathbf{a}}

\renewcommand{\bbf}{\mathbf{f}}

\newcommand{\bbn}{\mathbf{n}}

\newcommand{\bbx}{\mathbf{x}}
\newcommand{\bby}{\mathbf{y}}

\newcommand{\bsa}{{\bm {\alpha}}}
\newcommand{\bsb}{{\bm {\beta}}}
\newcommand{\balpha}{{\bm {\alpha}}}

\newcommand{\bmu}{{\bm {\mu}}}
\newcommand{\blambda}{{\bm {\lambda}}}
\newcommand{\bsigma}{{\bm {\sigma}}}
\newcommand{\bnu}{{\bm {\nu}}}

\newcommand{\Id}{\textbf{Id}}

\newcommand{\II}{\mathcal{I}}
\newcommand{\GG}{\mathcal{G}}

\newcommand{\hbbf}{\hat{\mathbf{f}}}
\newcommand{\DD}{\mathcal{D}}
\newcommand{\dg}{dG\xspace}
\newcommand{\pampa}{PamPa\xspace}
\newcommand{\dof}{\text{DoFs}}
\newcommand{\remi}[1]{{#1}}

\graphicspath{{figures/}{./}}

\title{Some new properties of an  Active flux type scheme: \pampa  }

\author{R\'{e}mi Abgrall $(*)$ , Yongle Liu $(*)$ and Philipp \"{O}ffner $(\dagger)$\\
$(*)$Institute of Mathematics, University of Z\"{u}rich, 8057 Z\"{u}rich, Switzerland \\
$(\dagger)$Institute of Mathematics, Clausthal University of Technology, \\D-38678 Clausthal-Zellerfeld, Germany\\
remi.abgrall@math.uzh.ch, yongle.liu@math.uzh.ch, philipp.oeffner@tu-clausthal.de}
\date{}

\usepackage{amsopn}
\DeclareMathOperator{\diag}{diag}

\usepackage{lineno}

\begin{document}

\maketitle

\begin{abstract}
In this paper, we provide a few new properties of Active Flux (AF)/Point-Average-Moment PolynomiAl-interpreted (\pampa) schemes. 
First, we show, in full generality, that the {AF/\pampa} schemes can be interpreted in such a way that the discontinuous Galerkin (dG) scheme is one of their building blocks. 
Secondly we provide  {intrinsic bound preserving properties of the current variant of \pampa}. This is also illustrated numerically. Last, we show, at least in one dimension, that the \pampa scheme has the summation by part (SBP) property.
\end{abstract}




\section{Introduction.}\label{sec1}
Since the seminal work of P.L. Roe and his students, \cite{AF1,AF2,AF3,EymannRoe2013,Maeng,He}, there has been a growing interest in the so-called Active Flux schemes, see e.g. \cite{AF5,Barsukow,AF4,Abgrall_AF,Abgrall2024_WBAF,BarsukowAbgrall,Pampa1D,BP_Pampa_VEM} { for hyperbolic conservation laws,
\begin{equation}
\label{eq:hyperbolic}
\dpar{\bbu}{t}+\text{div }\bbf(\bbu)=0, \bbx\in \Omega\subset\R^d, t\geq 0
\end{equation}
subjected to initial and boundary conditions. Bold letters will be used for vector valued functions. For example, the conserved variables will be denoted by bold letters ($\bbu$) when they belong to $\R^p$ with $p>1$, and non bold ones when $p=1$. The flux is $\bbf=(f_1, \ldots, f_d)$ is assumed to be at least $C^1$, and defined on $\mathcal{D}\subset \R^p$. This set is called the invariant domain, for example in the case of the Euler equations, it is defined by imposing a strictly positive density and internal energy.} 

These schemes evolve simultaneously the average values and point values on the boundary of elements, that can be quads, triangles and polygons in several space dimensions, and intervals in 1D. In these papers, several progress have been made: we are now able to handle non linear cases without spurious oscillations. Very high order schemes have been developed, see e.g. \cite{BarsukowAbgrall,barsukow2025generalizedactivefluxmethod}. Conformal and non conformal meshes can be used \cite{Calhoun2023-ms,BP_Pampa_VEM}.  Several very intriguing properties have been noticed, and in our opinion partially explained: the schemes, at least in the Cartesian version, seems to have very interesting properties for the low Mach flows
\cite{Barsukow}. 

 { What is missing is to understand these properties: this is the motivation of this paper. The present paper will  not  answer all the questions, but to provide some new links to more classical schemes, and to show that  some facts that were known in one dimension (see e.g. \cite{Pampa1D}) are not specific to one dimensional problems.

In this paper we are interested in three things. We first show the connection of the \pampa scheme with discontinuous Galerkin (\dg) schemes. In a way, the \pampa scheme is a \ldots continuous discontinuous Galerkin scheme. We show that for the 1D and multi D version of the scheme, for any order and any type of elements (interval, and simplex in several dimensions).

The second contribution is about bound preservation.
In \cite{BP_Pampa_VEM} it was noticed that the behavior of the average values is always smoother than that of the point values. In this paper, we show several intrinsic bound preserving properties on the average values.
}

{The last contribution is about summation by parts properties of the \pampa scheme in one dimension.}

{During the review process, we have been made aware of the reference \cite{connard} which has  a non empty overlap with our section \ref{sec2}. This paper has been sent to ArXiv at the same time as we submitted   posted this paper on ArXiv.  In section \ref{sec2}, we go above third order accuracy, or the one dimensional or Cartesian setting, however the main idea of connecting  the approximation of \eqref{eq:advec} to finite element method is essentially the same.}

{Our scheme has of course a strong connection to the Active Flux family of schemes. Indeed, one of us has been very much inspired by P.L. Roe's work, and in particular the references we have listed above, and also personal discussions. He also  had the opportunity to listen to Franco Brezzi's talks at the ICM in 2014. This talk was about Virtual Finite Elements (VEM), and is a subset of \cite{brezzi}. The connections were obvious, in particular how functions were approximated. For that reason, and because we do not use exact evolution operators, but the method of line, we decided \remi{not} to use the vocable of ``active flux'', but something else: our flux is inactive, and ``inactive flux'' would not have been very respectful. The naming of \pampa was introduced in \cite{Pampa1D}, and we have sticked to it since, even if the question of invariant domain preservation is not always central.}

\section{Reinterpretation of the \pampa scheme.}\label{sec2}
{In this section, we describe how the \pampa scheme can be reinterpreted starting from  the \dg method. 

In a first step, we show in one dimension and several dimensions that in each element, we can rewrite the \dg method using standard duality tools of linear algebra, in the case of a linear with constant coefficient hyperbolic. This is done for segments in one dimension, and simplex in two dimensions. The extension to 3D is straightforward. Using this, we show that the \pampa scheme of \cite{Abgrall_AF,Abgrall2023b,BarsukowAbgrall} is nothing more, when using a Runge-Kutta (RK) time procedure than, for each RK cycle, one step of \dg followed by a projection onto
 globally continuous approximations. The case of non linear problems is also described, and we discuss several projections.
 }

\subsection{Standard \pampa Scheme.}
We consider a tessellation of the 1-D spatial domain $\Omega$ in non overlapping elements $I_\jph=[x_j,x_{j+1}]$ with uniform size {$\dx_{j+1/2}=x_{j+1}-x_j$}. For sake of simplicity, we consider the advection equation:
\begin{equation}\label{eq:advec}
  u_t+au_x=0
\end{equation}
{with $a>0$. We will write the flux as $f(u)=au$.}

In the standard third-order \pampa (or the so-called generalized AF) scheme, the solution of \eqref{eq:advec} is approximated by a globally continuous finite element polynomial expansion\footnote{Following the finite element convention, we use the subscript ${\mathrm{h}}$ to indicate a finite element approximation of a variable.} $u_{\mathrm h}$  within each element $I_\jph$:
{\begin{equation}\label{eq:FEMapp}
  u^{j+1/2}_{\mathrm h}(x)=u_j \varphi_0(\xi)+\xbar{u}_\jph\varphi_1(\xi)+u_{j+1}\varphi_2(\xi),\quad \xi=\frac{x-x_j}{\dx_{j+1/2}},\quad x\in I_\jph,
\end{equation}  }
where the quadratic polynomial basis functions are 
\begin{equation*}
  \varphi_0=(1-\xi)(1-3\xi),\quad \varphi_1=6\xi(1-\xi), \quad \varphi_2=\xi(3\xi-2).
\end{equation*}

The \pampa scheme is, for all $j\in \Z$: the average values evolve with
\begin{subequations}\label{eq:PAMPA}
\begin{equation}\label{eq:PAMPA:ave}
  \dx_{j+1/2}\frac{\mathrm d\xbar{u}_{j+1/2}}{\mathrm dt}+a(u_{j+1}-u_j)=0,
  \end{equation}
  and the point values with
  \begin{equation}\label{eq:PAMPA:pt}
  \dx_{j+1/2}\frac{\mathrm d u_{j+1}}{\mathrm dt}+a\big( 2u_j+4u_{j+1}-6\xbar{u}_{j+1/2} \big)=0.
\end{equation}
\end{subequations}
{The rational behind the relation for $u_{j+1}$ is that one can approximate the spatial derivative or $u_{\mathrm h}$ at $x=x_{j+1}$ by using the approximation of $u$ either on $I_{j+1/2}$ of $I_{j+3/2}$ or combinations of the derivatives. Since $a>0$, we take the approximation using information from the downwind interval, i.e. $I_{j+1/2}$, so that
$$\dfrac{\mathrm d u^{j+1/2}_h}{\mathrm d x}(x_{j+1})=\dfrac{2u_j+4u_{j+1}-6\xbar{u}_{j+1/2}}{\Delta x_{j+1/2}}.$$}

{If $a$ were negative, the relation \eqref{eq:PAMPA:pt} would be replaced by
$$\dx_{j+3/2}\frac{\mathrm d u_{j+1}}{\mathrm dt}+a\big( 6\xbar{u}_{j+3/2}-4u_{j+1}-2u_{j+2}\big)=0$$
or equivalently by
$$\dx_{j+1/2}\frac{\mathrm d u_{j}}{\mathrm dt}+a\big( 6\xbar{u}_{j+1/2}-4u_{j}-2u_{j+1}\big)=0.$$}
Now, we reinterpret it as an \dg scheme. 

{The mass matrix  $\mathtt{M}$ given by $\mathtt{M}_{i,j}= \int_{x_j}^{x_{j+1}} \varphi_i(\xi) \varphi_j(\xi) \mathrm{d} \xi$ is 
\begin{equation}\label{eq:mass_dG}
  \mathtt{M}=\dx_{j+1/2}\begin{pmatrix}
        \frac{2}{15} & -\frac{1}{10} & -\frac{1}{30} \\
           -\frac{1}{10} & \frac{6}{5} & -\frac{1}{10} \\
           -\frac{1}{30} & -\frac{1}{10} & \frac{2}{15} 
       \end{pmatrix}
\end{equation}
and its inverse is
\begin{equation}\label{eq:invmass_dG}
  \mathtt{M}^{-1}=\frac{1}{\dx_{j+1/2}}\begin{pmatrix}
            9 & 1 & 3 \\
           1 & 1 & 1 \\
           3 & 1 & 9 
       \end{pmatrix}
\end{equation}
in each  element $I_{j+\frac{1}{2}}$.}
On the other hand, we have the following update procedure
\begin{equation}\label{eq:flux_dG}
\begin{aligned}
  &\int_{I_\jph}\varphi_0\big (\frac{x-x_j}{\dx_{j+1/2}}\big ) u_{\mathrm h}'(x)\;{\mathrm d}x=\xbar{u}_\jph-\frac{u_j+u_{j+1}}{2},\\
  &\int_{I_\jph}\varphi_2\big (\frac{x-x_j}{\dx_{j+1/2}}\big ) u_{\mathrm h}'(x)\;{\mathrm d}x=u_{j+1}-u_j,\\ 
  &\int_{I_\jph}\varphi_1 \big (\frac{x-x_j}{\dx_{j+1/2}}\big )u_{\mathrm h}'(x)\;{\mathrm d}x=\frac{u_j+u_{j+1}}{2}-\xbar{u}_\jph.
\end{aligned}
\end{equation}
So that, we get in total 
\begin{equation}\label{eq:dG_form1}
  \mathtt{M}\frac{\mathrm d}{\mathrm dt}\begin{pmatrix}
                                 u_j \\
                                 \xbar{u}_\jph \\
                                 u_{j+1} 
                               \end{pmatrix}
  +a\begin{pmatrix}
      \xbar{u}_\jph-\frac{u_j+u_{j+1}}{2} \\
      u_{j+1}-u_j \\
      \frac{u_j+u_{j+1}}{2}-\xbar{u}_\jph 
    \end{pmatrix}=0,
\end{equation}
which can be further written as, using \eqref{eq:invmass_dG}:
\begin{equation}\label{eq:dG_form2}
  \frac{\mathrm d}{\mathrm dt}\begin{pmatrix}
                                 u_j \\
                                 \xbar{u}_\jph \\
                                 u_{j+1} 
                               \end{pmatrix}
  +\frac{a}{{\dx_{j+1/2}}}\begin{pmatrix}
      6\xbar{u}_\jph-4u_j-2u_{j+1} \\
      u_{j+1}-u_j \\
      2u_j+4u_{j+1}-6\xbar{u}_\jph 
    \end{pmatrix}=0.
\end{equation}

We also note that
\begin{equation*}
  \frac{a}{\dx_{j+1/2}}\begin{pmatrix}
      6\xbar{u}_\jph-4u_j-2u_{j+1} \\
      u_{j+1}-u_j \\
      2u_j+4u_{j+1}-6\xbar{u}_\jph 
    \end{pmatrix}=\begin{pmatrix}
                    f'(u_j) \\
                    f(u_{j+1})-f(u_j)\\
                    f'(u_{j+1}) 
                  \end{pmatrix}.
\end{equation*}
{In the general case of $a\in \R$, we can  interpret \pampa as follows:}
\begin{enumerate}
  \item Run \dg as here,
  \item  Drop the update of $u_j$ 
  \begin{itemize}
      \item if $a>0$: 
  coming from the interval $I_\jph$ and keep the internal DoFs computed from dG. 
  \item If $a<0$, coming from the interval $I_\jmh$ and keep the internal DoFs computed from dG.
  \end{itemize}
\end{enumerate}
\remi{This justifies the naming of "Continuous \dg": starting at each Runge-Kutta cycle by a globally continuous approximation, we first apply a standard \dg procedure, taking into account the simplifications introduced by the global continuity. After this cycle, each point value degree of freedom is multi-valued: there is, a priori, one value per element that share this degree of freedom. Then we need to project back on the original space, this is done by taking into account upwinding. This procedure is generalized in several dimensions in the next sections}

\subsection{High order \pampa scheme.}
In the high order case, we take $m_i(x)=\big (\frac{x-x_j}{\dx_{j+1/2}}\big)^i$ in the interval $I_\jph$ and define the degrees of freedom (DoFs) for degree $k$ (where the formal order of the scheme is $k+1$) as
\begin{equation}\label{eq:DoFs}
  u_0\approx u(x_j),\quad u_1\approx u(x_{j+1}),\quad u_{l+2}=\int_{I_\jph}m_{l}(x)u(x)\;{\mathrm d}x,~\text{for} ~0\leq l\leq k-2.
\end{equation}
Note that in \eqref{eq:DoFs}, we have changed the ordering of the corresponding coefficients  due to the high-order approach.
We denote the linear forms as follows:
\begin{equation*}
  \langle \theta_l,u\rangle=u_l, \quad 0\leq l\leq k
\end{equation*}
where $\theta_l$ is a bounded linear functional. We denote the dual basis in $\P^k$ by  $\{\varphi_q\in\mathbb{P}^k\}$, $q=0, \ldots, k$ such that $\langle \theta_l,\varphi_q \rangle=\delta_{lq}$. In $I_\jph$, $u(x)$ is approximated as
\begin{equation}\label{eq:FEMu}
  \uh(x)=\sum_{q=0}^{k}u_q\varphi_q(x).
\end{equation}

In the  \pampa method, we first compute the point value evolution equations
\begin{subequations}\label{eq:PAMPA_HO}
\begin{equation}
\begin{split}\label{eq:PAMPA_HO:pt}
  \frac{\mathrm d u_0}{\mathrm dt}&+a\frac{\mathrm d \uh}{\mathrm dx}(x_j)=0,\\
  \frac{\mathrm d u_1}{\mathrm dt}&+a\frac{\mathrm d \uh}{\mathrm dx}(x_{j+1})=0,\\
  \end{split}
  \end{equation}
  and for the moments {for} $0\leq l\leq k-2$, we write
  \begin{equation}\label{eq:PAMPA_HO:ave}
  \frac{\mathrm d u_{l+2}}{\mathrm dt}+\int_{I_\jph}a\, m_{l}(x)\,\frac{\mathrm d \uh}{\mathrm dx}\;\mathrm dx=0.
\end{equation}
\end{subequations}
Using the same basis for $\mathbb{P}^k$, the \dg method writes
\begin{equation*}
  \mathtt{M}\frac{\mathrm d \bbU}{\mathrm dt}+\bbF=0
\end{equation*}
with $\mathtt{M}_{l\kappa}=\int_{I_\jph}\varphi_l(x)\varphi_\kappa(x)\;{\mathrm d}x$ and 
\begin{equation*}
  \bbF_\kappa=\int_{I_\jph} a\varphi_\kappa(x)\frac{\mathrm d\uh}{\mathrm dx}\;{\mathrm d}x.
\end{equation*} 
The question is to compare them.

For this, using Riesz theorem, we can write $\langle\theta_l,u\rangle=(\psi_l,u)$ where 
\begin{equation*}
  (u,v)=\int_{I_\jph}u(x)v(x)\;{\mathrm d}x.
\end{equation*} and $\psi_l\in \P^k$ again.
Clearly, for $l\geq 0$, $\psi_{l+2}=m_{l}$. For $l=0,1$ (corresponding to the vertices $x_j$ and $x_{j+1}$), we also write $\psi_l$  as linear combinations of the $\varphi_q$s: 
\begin{equation*}
  \psi_l=\sum_{q=0}^{k}a_{lq}\varphi_q
\end{equation*}
and from the orthogonality condition {$\delta_{lp}=\langle\theta_l,\varphi_p\rangle=(\psi_l,\varphi_p)$},
\begin{equation*}
\delta_{lp}{=(\psi_l,\varphi_p)=\sum_{q=0}^ka_{lq}\big (\varphi_q,\varphi_p\big )}=\sum_{q=0}^{k}a_{lq}\mathtt{M}_{qp}
\end{equation*}
that we interpret as saying that the inverse of $\mathtt{M}$ is the matrix $(a_{lq})$.

Then we compute $\mathtt{M}^{-1}\bbF$ in a component-wise manner, and get
\begin{equation}\label{eq_General}
  \begin{aligned}
  (\mathtt{M}^{-1}\bbF)_p&=\sum_{l=0}^{k}a_{pl}F_l=\sum_{l=0}^{k}a_{pl}\int_{I_\jph}a\varphi_l\frac{\mathrm d \uh}{\mathrm dx}{\mathrm d}x\\
  &=a\; \int_{I_\jph}\Big(\sum_{l=0}^{k}a_{pl}\varphi_l\Big)\frac{\mathrm d \uh}{\mathrm dx}{\mathrm d}x\\
  &=a\;\int_{I_\jph}\psi_p\frac{\mathrm d \uh}{\mathrm dx}{\mathrm d}x.
  \end{aligned}
\end{equation}
{The last step is to show that 
$$a\int_{I_\jph}\psi_p\frac{\mathrm d \uh}{\mathrm dx}{\mathrm d}x=a\dfrac{\mathrm{d}\uh}{\mathrm dx}(x_{j+p})$$ for $p=0,1$, i.e. the two linear forms associated with the points values.
This is indeed the fact if $a\in \R$ is a constant and since $\frac{\mathrm d \uh}{\mathrm dx}\in \P^k$ if $\uh\in \P^k$.}

This shows that we recover the \pampa that we can again interpret as 
\begin{enumerate}
  \item Run \dg as here,
  \item Drop the update of $u_j$ 
  \begin{itemize}
      \item if $a>0$: 
  coming from the interval $I_\jph$ and keep the internal DoFs computed from dG. 
  \item If $a<0$, coming from the interval $I_\jmh$ and keep the internal DoFs computed from dG.
  \end{itemize}
\end{enumerate}

{ \subsection{\pampa for triangles.}\label{sec:triangle}
We consider the scalar problem $u_t+\text{ div }\bbf=0$ and the case of a triangle $K$. We want that the approximation space contains $\P^k$ for the accuracy.  We denote the vertices by $\bba_i$, $i=1,2,3$ and $\lambda_i$ the barycentric coordinates of $K$. The DoFs are these vertices and the other $k-1$ Gauss--Lobatto points on the three edges, and the $k(k-1)/2$ additional moments \cite{BP_Pampa_VEM}. All of these DoFs assemble the approximation space with dimension $N_{\dof}=3+3(k-1)+k(k-1)/2=3k+k(k-1)/2$. 

The dimension of $\P^k$ is $(k+1)(k+2)/2$. The number of Lagrange points on each edge, taking into account the vertices as Lagrange points, is $3k$. We have
$$3k+k(k-1)/2-3k=\frac{k(k-1)}{2}={\text{dim }\P^{k-2}}.$$
These equalities can be rephrased as follows. We set $B_\bmu=c_\bmu \lambda_1^{\mu_1}\lambda_2^{\mu_2}\lambda_3^{\mu_3}$ the B\'ezier polynomials. For the sake of simplicity, we will write
$$\lambda_1^{\mu_1}\lambda_2^{\mu_2}\lambda_3^{\mu_3}=\blambda^\bmu.$$
The multi-index $\bmu$ is $\bmu=(\mu_1,\mu_2,\mu_3)$. We set $\vert \bmu\vert=\mu_1+\mu_2+\mu_3$.  The coefficients $c_\bmu$ are binomial coefficients so that the $B_\bmu$s can be \remi{seen} as the terms in the development of $(\lambda_1+\lambda_2+\lambda_3)^k$.  We define the two vector spaces:
$$V_1=\text{span}\{ B_\bmu, \vert \bmu\vert=k \text{ and } \mu_1\mu_2\mu_3=0\}$$
and
$$V_2=\text{span}\{ B_\bmu, \vert \bmu\vert=k \text{ and } \mu_1\mu_2\mu_3>0\}.$$
Clearly, $\P^k\subset V_1\bigoplus V_2$. The elements of $V_2$ vanish on the boundary of $K$. For the sake of clarity, we will denote the generator of $V_2$ with a specific notation
$$\psi_{\bmu}=\lambda_1\lambda_2\lambda_3 B_{\bmu}$$ 
with\begin{itemize}
\item if $1\leq k\leq 2$, $B_{\bmu}=1$,
\item and for $k\geq 3$, $B_{\bmu}=\blambda^{\bmu}$ with 
$$\vert \bmu\vert=k-2.$$
\end{itemize}
It is clear that $\text{dim }V_1=3k$ and $\text{dim }V_2=\text{ dim }\P^{k-2}$. We define
$$\mathcal{I}_k=\left \{ \begin{array}{ll}
\{\bmu, \vert \bmu\vert=k-2\}&\text{ for }k>2\\
\{(0,0,0)\}  &\text{ for }k\leq 2.
\end{array}
\right .$$
The cardinal of $\mathcal{I}_k$ is {$\text{ dim }\P^{k-2}$}. 
We will often write $b=\lambda_1\lambda_2\lambda_3$.

\bigskip

\noindent We introduce the following polynomial spaces:
\begin{itemize}
\item For $k\leq 2$, $V=\P^k\bigoplus b\R$,
\item and for $k>2$, $V=\P^k \bigoplus b\P^{k-2}$
\end{itemize}
We introduce the following degrees of freedom:\begin{subequations}\label{linear forms}
\begin{itemize}
\item For the Gauss-Lobatto points $\sigma$ on the boundary of $K$,
\begin{equation}\langle \delta_\sigma, p\rangle =p(\sigma).\label{linear form points}\end{equation}
The set of Gauss-Lobatto points is denoted by $\mathcal{G}$ and contains $3k$ points.
\item In the interior,
\begin{equation}\langle m_{\bmu}, p\rangle =\frac{1}{\vert K\vert}\int_K \blambda^{\bmu}(\bbx)p(\bbx)\; \mathrm d\bbx\label{linear forms moment}\end{equation}
for $\bmu\in \mathcal{I}_k$.
\end{itemize}\end{subequations}
They are linear forms on $V$.
\begin{lemma}
The linear forms \eqref{linear form points}-\eqref{linear forms moment} are uni-solvant on $V$.
\end{lemma}
\begin{proof}
Let $\alpha_\sigma$, $\sigma\in \mathcal{G}$ and $\alpha_{\bmu'}$, $\bmu'\in \mathcal{I}_k$ be real numbers such that for all $p\in V$
\begin{equation}
\label{test}\sum_{\sigma\in \mathcal{G}}\alpha_{\sigma} \langle\delta_\sigma,p\rangle+\sum_{\bmu'\in \mathcal{I}_k}\alpha_{\bmu'}\langle m_{\bmu'},p\rangle=0.
\end{equation}
We first consider the Lagrange polynomials $L_\sigma$  at  $\sigma$ and define
\begin{equation}
\label{basis function}
\varphi_\sigma=L_\sigma-\sum\limits_{\bmu\in \mathcal{I}_k}
\frac{m_{\bmu}(L_\sigma)}{m_{\bmu}(\psi_{\bmu})}\psi_{\bmu}
\end{equation}
Since $\psi_{\bmu}$ vanishes on $\partial K$ and $L_\sigma$ is a Lagrange polynomial for the Gauss-Lobatto points, we have
$$\langle \delta_{\sigma'},   \varphi_{\sigma}\rangle=\delta_{\sigma}^{\sigma'},$$
and by construction, for all $\bmu\in \mathcal{I}_k$, 
$$\langle m_\bmu, \varphi_{\sigma}\rangle=0.$$
Using this to test \eqref{test}, we see that $\alpha_\sigma=0$ for $\sigma\in \mathcal{G}$.

Then we test with $p=\psi_\bmu$, $\bmu\in \mathcal{I}_k$, that is 
\begin{equation}
\label{basis functions2}\sum_{\bmu'\in \mathcal{I}_k}\alpha_{\bmu'}\int_K  \lambda_1\lambda_2\lambda_3 \blambda^{\bmu'}\blambda^\bmu\; \mathrm d\bbx=0.
\end{equation}
The question reduces to that of the invertibility of the $\vert \II_k\vert \times \vert \II_k\vert$ symmetric matrix
$$A=\begin{pmatrix}\int_K  \lambda_1\lambda_2\lambda_3 \blambda^{\bmu'}\blambda^\bmu\; \mathrm d\bbx
\end{pmatrix}.$$
Taking $X\in \R^{\vert \II_k\vert}$, we see that
$$X^TAX=\int_K \lambda_1\lambda_2\lambda_3 \varphi^2(\bbx)\; \mathrm d\bbx$$
with
$$\varphi=\sum_{\bmu\in \II_k}X_{\bmu} \blambda^\bmu.$$
Since $\{\blambda^\bmu\}_{\bmu\in \II_k}$ is \remi{ a free family}, the matrix is positive definite. Thus the forms are unisolvant.
\end{proof}

\bigskip

The second step is to compute the dual basis of $\{\delta_\sigma, m_\bmu\}_{\sigma, \bmu}$, i.e. elements $\{\phi_\sigma, \phi_\bmu\}_{\sigma,\bmu}$ of $V$ such that for all $\sigma$,
$$\forall  \sigma', \langle \delta_\sigma, \varphi_{\sigma'}\rangle=\delta_\sigma^{\sigma'}\text{ and }\forall \bmu, \langle \delta_\sigma, \varphi_\bmu\rangle=0$$
and for all $\bmu$,
$$\forall \sigma,  \langle m_{\bmu}, \varphi_{\sigma}\rangle=0 \text{ and }\forall \bmu', \langle m_{\bmu}, \varphi_{\bmu'}\rangle=\delta_{\bmu}^{\bmu'}.$$

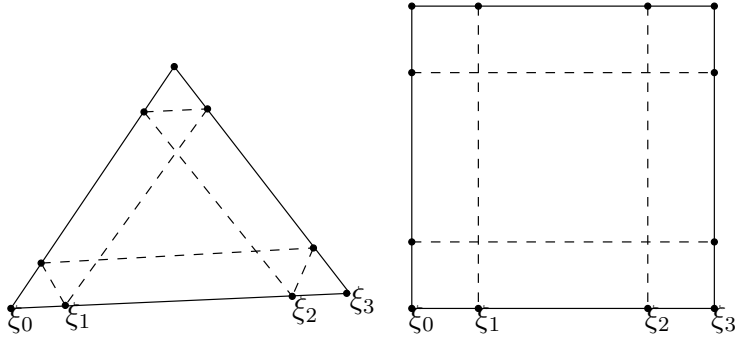
\begin{figure}
\begin{center}
\begin{tabular}{cc}
{\pgfkeys{/pgf/fpu/.try=false}%
 \ifdim\dimen1=0pt\ifdim\dimen3=0pt\dimen1=1000sp\dimen3\dimen1
  \else\dimen1\dimen3\fi\else\ifdim\dimen3=0pt\dimen3\dimen1\fi\fi
\begin{tikzpicture}[x=+\dimen1, y=+\dimen3]
{\ifx\XFigu\undefined\catcode`\@11
\def\temp{\alloc@1\dimen\dimendef\insc@unt}\temp\XFigu\catcode`\@12\fi}
\XFigu3946sp
\ifdim\XFigu<0pt\XFigu-\XFigu\fi
\tikzset{inner sep=+0pt, outer sep=+0pt}
\pgfsetfillcolor{black}
\pgftext[base,left,at=\pgfqpointxy{9825}{-8100}] {\fontsize{12}{14.4}\usefont{T1}{ptm}{m}{n}$\xi_3$}
\pgfsetlinewidth{+7.5\XFigu}
\pgfsetcolor{black}
\filldraw  (9750,-7725) circle [radius=+75];
\filldraw  (5475,-2100) circle [radius=+75];
\filldraw  (2175,-6975) circle [radius=+75];
\filldraw  (2775,-8025) circle [radius=+75];
\filldraw  (4725,-3225) circle [radius=+75];
\filldraw  (6300,-3150) circle [radius=+75];
\filldraw  (8925,-6600) circle [radius=+75];
\filldraw  (8400,-7800) circle [radius=+75];
\draw (1425,-8100)--(5475,-2100)--(9825,-7725)--(1425,-8100);
\pgfsetdash{{+60\XFigu}{+60\XFigu}}{++0pt}
\draw (6300,-3150)--(2775,-8025);
\draw (4725,-3225)--(8400,-7800);
\draw (2175,-6975)--(8925,-6600);
\draw (2175,-6975)--(2775,-8025);
\draw (8925,-6600)--(8400,-7800);
\draw (4725,-3225)--(6300,-3150);
\pgftext[base,left,at=\pgfqpointxy{1350}{-8550}] {\fontsize{12}{14.4}\usefont{T1}{ptm}{m}{n}$\xi_0$}
\pgftext[base,left,at=\pgfqpointxy{2775}{-8475}] {\fontsize{12}{14.4}\usefont{T1}{ptm}{m}{n}$\xi_1$}
\pgftext[base,left,at=\pgfqpointxy{8400}{-8325}] {\fontsize{12}{14.4}\usefont{T1}{ptm}{m}{n}$\xi_2$}
\pgfsetdash{}{+0pt}
\filldraw  (1425,-8100) circle [radius=+75];
\end{tikzpicture} }%
&{\pgfkeys{/pgf/fpu/.try=false}%
 \ifdim\dimen1=0pt\ifdim\dimen3=0pt\dimen1=1000sp\dimen3\dimen1
  \else\dimen1\dimen3\fi\else\ifdim\dimen3=0pt\dimen3\dimen1\fi\fi
\begin{tikzpicture}[x=+\dimen1, y=+\dimen3]
{\ifx\XFigu\undefined\catcode`\@11
\def\temp{\alloc@1\dimen\dimendef\insc@unt}\temp\XFigu\catcode`\@12\fi}
\XFigu3946sp
\ifdim\XFigu<0pt\XFigu-\XFigu\fi
\tikzset{inner sep=+0pt, outer sep=+0pt}
\pgfsetfillcolor{black}
\pgftext[base,left,at=\pgfqpointxy{9525}{-8850}] {\fontsize{12}{14.4}\usefont{T1}{ptm}{m}{n}$\xi_3$}
\pgfsetlinewidth{+7.5\XFigu}
\pgfsetcolor{black}
\filldraw  (2100,-6750) circle [radius=+75];
\filldraw  (2100,-2550) circle [radius=+75];
\filldraw  (2100,-900) circle [radius=+75];
\filldraw  (3750,-900) circle [radius=+75];
\filldraw  (7950,-900) circle [radius=+75];
\filldraw  (9600,-900) circle [radius=+75];
\filldraw  (9600,-2550) circle [radius=+75];
\filldraw  (9600,-6750) circle [radius=+75];
\filldraw  (9600,-8400) circle [radius=+75];
\filldraw  (7950,-8400) circle [radius=+75];
\filldraw  (3750,-8400) circle [radius=+75];
\draw (2100,-900)--(2100,-8400)--(9600,-8400)--(9600,-900)--(2100,-900);
\pgfsetdash{{+60\XFigu}{+60\XFigu}}{++0pt}
\draw (3750,-900)--(3750,-8400);
\draw (7950,-900)--(7950,-8400);
\draw (2100,-2550)--(9600,-2550);
\draw (2100,-6750)--(9600,-6750);
\pgftext[base,left,at=\pgfqpointxy{2025}{-8850}] {\fontsize{12}{14.4}\usefont{T1}{ptm}{m}{n}$\xi_0$}
\pgftext[base,left,at=\pgfqpointxy{3675}{-8850}] {\fontsize{12}{14.4}\usefont{T1}{ptm}{m}{n}$\xi_1$}
\pgftext[base,left,at=\pgfqpointxy{7875}{-8850}] {\fontsize{12}{14.4}\usefont{T1}{ptm}{m}{n}$\xi_2$}
\pgfsetdash{}{+0pt}
\filldraw  (2100,-8400) circle [radius=+75];
\end{tikzpicture}}%
\end{tabular}
\end{center}\label{fig:1}
\caption{Interpolation points for the triangular and quadrangular, cubic case.}
\end{figure}

We have the following lemma:
\begin{lemma}[Dual basis]\label{lemma:dual_basis}
The dual basis of the forms $\delta_\sigma, m_{\bmu}$ consists in the elements of $V$ defined by:
\begin{itemize}
\item Associated to the Gauss-Lobatto points, the family $\{\varphi_\sigma\}$ defined by \eqref{basis function}. 
\item Associated to the moments: the elements $\varphi_{\bmu}$ defined by
$$\varphi_{\bmu}=\lambda_1\lambda_2\lambda_3\big (\sum_{\bnu\in \II_k}a_{\bmu\bnu}\lambda^{\bnu}\big )$$
where the $\{a_{\bmu\bnu}\}$ are solution of the linear system
$$\balpha\big ( a_{\bmu\bnu}\big )=\big( \delta_{\bmu}^{\bnu}\big)=\Id_{\vert \II_k\vert\times\vert\II_k\vert}$$
with 
$$\alpha_{\bmu,\bnu}=2\vert K\vert \frac{(\mu_1+\nu_1+1)!(\mu_2+\nu_2+1)(\mu_3+\nu_3+1)!}
{(2k+1)!}.
$$ This matrix is positive definite.
\item The coefficients $a_{\bmu\bnu}$ only depend on $\bmu+\bnu$, we write $a_{\bmu\bnu}=a_{\bmu+\bnu}$.
\end{itemize}
\end{lemma}
\begin{proof}
We need to  define the $\varphi_\bmu$s.
Since $\langle \delta_\sigma, \varphi_\bmu\rangle=0$ for all $\sigma$, we see that
$$\varphi_{\bmu}=\lambda_1\lambda_2\lambda_3\big (\sum_{\bnu\in \II_k}a_{\bmu\bnu}\lambda^{\bnu}\big )$$
and from the second set of family, that for all $\bnu'$, 
$$\sum_{\bnu\in \mathcal{I}_k}\alpha_{\bmu\bnu}\int_K  \lambda_1\lambda_2\lambda_3 \blambda^{\bnu}\blambda^{\bnu'}\; \mathrm d\bbx=\delta_{\bmu}^{\bnu'}.$$
\remi{We have} 
$$\int_K  \lambda_1\lambda_2\lambda_3 \blambda^{\bmu}\blambda^{\bnu}\; \mathrm d\bbx=
2\vert K\vert \frac{(\mu_1+\nu_1+1)!(\mu_2+\nu_2+1)(\mu_3+\nu_3+1)!}
{(2k+1)!}
$$
because $\bmu, \bnu\in \II_k$ so that $2+\sum_i\mu_i+\sum_i\nu_i+3=2(k-2)+5=2k+1$. For the same reason as above, this matrix is invertible because it is positive definite. 

We also see that since  $\alpha_{\bmu,\bnu}$ depends only on $\bmu+\bnu$, using for example Cramer's formula, we get that $a_{\bmu\bnu}$ only depends on $\bmu+\bnu$.
\end{proof}

For example for $k=3$ (which is the first non trivial example), {we have, for the moment $ m_\bmu$ with $\bmu=(\mu_1,\mu_2,\mu_3)$ where the $\mu_1+\mu_2+\mu_3=1$, all the $\mu_j$s are all equal to $0$ except for one which is equal to $1$. This leads to 
$$\varphi_\bmu=\lambda_1\lambda_2\lambda_3(a_{1}^k\lambda_1+a_{2}^k\lambda_2+a_3^k\lambda_3).$$} The matrix 
$$A=\frac{1}{6!}\begin{pmatrix}
6 & 4& 4\\
4&6&4\\
4&4&6\end{pmatrix}$$ is circulant. 
{Since} $X{=}(a_1,a_2,a_3)^T=
(\tfrac{1800}{7},-\frac{720}{7},-\frac{720}{7})^T$ is the solution of
$$A X=\begin{pmatrix} 1\\ 0 \\0\end{pmatrix}$$
the polynomials are
\begin{itemize}
    \item Moment $\lambda_1$: $\varphi_{(1,0,0)}=\lambda_1\lambda_2\lambda_3(a_1\lambda_1+a_2\lambda_2+a_3\lambda_3)$.
    \item Moment $\lambda_2$: $\varphi_{(0,1,0)}=\lambda_1\lambda_2\lambda_3(a_3\lambda_1+a_1\lambda_2+a_2\lambda_3)$.
    \item Moment $\lambda_3$: $\varphi_{(0,0,1)}=\lambda_1\lambda_2\lambda_3(a_2\lambda_1+a_3\lambda_2+a_1\lambda_3)$.
    \end{itemize}

\bigskip

From Riesz theorem, we know that for any linear form $\varphi^\star$ defined on $V$, there exists an element $\varphi\in V$ such that for all $p\in V$,
$$\langle \varphi^\star, p\rangle=\frac{1}{\vert K\vert} \int_K \varphi(\bbx)p(\bbx)\; \mathrm d\bbx$$
because 
$$(p,q)=\frac{1}{\vert K\vert} \int_K p(\bbx)q(\bbx)\; \mathrm d\bbx$$
is a scalar product on $V$.

From its definition, $m_{\bmu}$ for $\bmu\in \II_k$ is associated to $\blambda^\bmu$ since
$$\langle m_\bmu, p\rangle=\frac{1}{\vert K\vert}\int_K\blambda^\bmu p\; \mathrm d\bbx.$$
 We are interested in looking at $\psi_\sigma$ such that for all $p\in V$,
$$\langle \delta_{\sigma},p\rangle=\frac{1}{\vert K\vert}\int_K \psi_{\sigma} p\; \mathrm d\bbx.$$
We write
\begin{subequations}
\label{dualdual}
\begin{equation}\label{dualdual:1}\psi_{\sigma}=\sum\limits_{\sigma'\in \GG}\alpha_{\sigma\sigma'}\varphi_{\sigma'}+\sum\limits_{\bmu'\in \II_k}\alpha_{\sigma\bmu' }\varphi_{\bmu'}
\end{equation}
and
\begin{equation}\label{dualdual:2}\blambda^\bmu=\sum\limits_{\sigma'\in \GG}\alpha_{\bmu\sigma'}\varphi_{\sigma'}+\sum\limits_{\bmu'\in \II_k}\alpha_{\bmu\bmu' }\varphi_{\bmu'}
\end{equation}
and get
$$\vert K\vert \delta_{\sigma\sigma'}=\sum\limits_{\sigma'\in \GG}\alpha_{\sigma\sigma'}\int_K\varphi_{\sigma}\varphi_{\sigma'}\; \mathrm d\bbx+\sum\limits_{\bmu'\in \II_k}\alpha_{\sigma\bmu' }\int_K\varphi_{\sigma} \varphi_{\bmu'}\; \mathrm d\bbx,$$
$$0=\sum\limits_{\sigma'\in \GG}\alpha_{\sigma\sigma'}\int_K\varphi_{\bmu}\varphi_{\sigma'}\; \mathrm d\bbx+\sum\limits_{\bmu'\in \II_k}\alpha_{\sigma\bmu' }\int_K\varphi_{\bmu} \varphi_{\bmu'}\; \mathrm d\bbx,$$
$$0=\sum\limits_{\sigma'\in \GG}\alpha_{\bmu\sigma'}\int_K\varphi_\sigma\varphi_{\sigma'}\; \mathrm d\bbx+\sum\limits_{\bmu'\in \II_k}\alpha_{\bmu\bmu' }\int_K\varphi_\sigma \varphi_{\bmu'}\; \mathrm d\bbx,$$
$$\vert K\vert \delta_{\bmu\bmu'}=\sum\limits_{\sigma'\in \GG}\alpha_{\bmu\sigma'}\int_K\varphi_\bmu\varphi_{\sigma'}\; \mathrm d\bbx+\sum\limits_{\bmu'\in \II_k}\alpha_{\bmu\bmu' }\int_K\varphi_{\bmu} \varphi_{\bmu'}\; \mathrm d\bbx$$
that is
\begin{equation}\label{dualdual:3}
\begin{split}
\vert K\vert &\begin{pmatrix}
\Id_{(3k)\times(3k)} & \mathbf{0}_{(3k)\times \vert \II_k\vert}\\
\mathbf{0}_{\vert \II_k\vert \times(3k)}&\Id_{\vert \II_k\vert\times \vert \II_k\vert}\end{pmatrix}
=
\begin{pmatrix}
(\alpha_{\sigma\sigma'})_{(3k)\times(3k)} & (\alpha_{\sigma\bmu })_{(3k)\times \vert \II_k\vert}\\
(\alpha_{\sigma\sigma'})_{\vert \II_k\vert \times(3k)}& (\alpha_{\mu\bmu' })_{\vert \II_k\vert\times \vert \II_k\vert}
\end{pmatrix}\\
&\qquad \qquad\times\begin{pmatrix}
\big ( \int_K\varphi_{\sigma}\varphi_{\sigma'}\; \mathrm d\bbx\big )_{(3k)\times(3k)}& \big ( \int_K\varphi_{\sigma}\varphi_{\bmu'}\; \mathrm d\bbx\big )_{(3k)\times \vert \II_k\vert}\\
\big ( \int_K\varphi_{\bmu}\varphi_{\sigma}\; \mathrm d\bbx\big )_{\vert \II_k\vert \times(3k)}& \big ( \int_K\varphi_{\bmu}\varphi_{\bmu'}\; \mathrm d\bbx\big )_{\vert \II_k\vert\times \vert \II_k\vert}
\end{pmatrix}.
\end{split}
\end{equation}
\end{subequations}
i.e. the matrix of the coordinates of the basis $\{\psi_{\bsigma},{ \blambda}^{\bmu} \}$ is the inverse of the mass matrix (up-to a factor $1/\vert K\vert$).

\paragraph{Application to the convection problem with constant advection speed.}
Then we extend what has been done in one dimension. We start from the problem
$$\dpar{u}{t}+\bba\cdot \nabla u=0$$ in $K$, assuming an initial condition in $V$. At $t=0$, since $\bba$ is constant, and from the definition of the elements of $V$, $\bba\cdot\nabla u\in V$, and then we have initially
\begin{equation}
\begin{split}
\dpar{}{t}\langle \delta_\sigma, u\rangle&+\langle \delta_\sigma, \bba\cdot \nabla u\rangle=0\\
\dpar{}{t}\langle m_\bmu, u\rangle&+\langle m_\bmu, \bba\cdot \nabla u\rangle=0
\end{split}
\end{equation}
that we can equivalently write as 
\begin{equation}
\begin{split}
\dpar{}{t}\int_K\psi_\sigma u\; \mathrm d\bbx&+\int_K \psi_\sigma  \bba\cdot \nabla u\; \mathrm d\bbx=0\\
\dpar{}{t}\int_K \blambda^\mu u\; \mathrm d\bbx&+\int_K \blambda^\bmu  \bba\cdot \nabla u\; \mathrm d\bbx=0.\\
\end{split}
\end{equation}
Then using \eqref{dualdual}, and in particular \eqref{dualdual:3},  we get
\begin{equation}\label{dual:linearadvection}
M\dfrac{\mathrm d}{\mathrm dt}\begin{pmatrix}
\int_K\varphi_\sigma u\; \mathrm d\bbx\\
\int_K\varphi_\bmu u\; \mathrm d\bbx.
\end{pmatrix}
+ \begin{pmatrix}
\int_K \varphi_\sigma  \bba\cdot \nabla u\; \mathrm d\bbx\\
\int_K \varphi_\bmu  \bba\cdot \nabla u\; \mathrm d\bbx
\end{pmatrix}=0
\end{equation}
\begin{rmk}
If, instead of writing the time continuous version of the problem, we would have had started from a semi-discretisation in time, for example the explicit Euler forward scheme,
$$u^{n+1}=u^n+\Delta t \bba\cdot \nabla u^n,$$
we would have got
\begin{equation*}
M\bigg (\begin{pmatrix}
\int_K\varphi_\sigma u\; \mathrm d\bbx\\
\int_K\varphi_\bmu u\; \mathrm d\bbx
\end{pmatrix}^{n+1}-\begin{pmatrix}
\int_K\varphi_\sigma u\; \mathrm d\bbx\\
\int_K\varphi_\bmu u\; \mathrm d\bbx
\end{pmatrix}^{n}\bigg )
+ \dt\begin{pmatrix}
\int_K \varphi_\sigma  \bba\cdot \nabla u\; \mathrm d\bbx\\
\int_K \varphi_\bmu  \bba\cdot \nabla u\; \mathrm d\bbx
\end{pmatrix}=0.
\end{equation*}
This is maybe a more rigorous way to proceed.
\end{rmk}

The method can be reinterpreted as in one dimension. We start from 
the \dg approximation, setting
$$u_{\mathrm h}=\sum_{\sigma}\langle\delta_\sigma, u\rangle\varphi_\sigma+\sum_{\bmu} \langle m_\bmu,u\rangle\phi_\bmu,$$ let this evolve in time by
$$\mathtt{M}\dfrac{\mathrm d}{\mathrm dt}\bbU+\bbF=0$$
with
$$\bbU=\big(\{\langle\delta_\sigma,u\rangle\}, \{\langle m_\bmu, u\rangle\}\big)^T$$ and 
$\bbF=(\{F_\sigma\}, \{F_\bmu\})^T$ with
$$F_\sigma=-\int_K \nabla \varphi_\sigma\cdot \bbf(u_{\mathrm h})\;\mathrm d\bbx+\oint_{\partial K}\varphi_\sigma \bbf(u_{\mathrm h})\cdot \bbn \; \mathrm d\gamma$$ and
$$F_\bmu=-\int_K \nabla \varphi_\bmu\cdot \bbf(u_{\mathrm h})\;\mathrm d\bbx+\oint_{\partial K}\varphi_\bmu \bbf(u_{\mathrm h})\cdot \bbn \; \mathrm d\gamma$$
and then we get
$$\dfrac{\mathrm dU}{\mathrm dt}+M^{-1} \begin{pmatrix} F_\sigma\\ F_\bmu\end{pmatrix}=0.$$

    \begin{rmk}[Scheme for quadrangles]
    The calculations are the same once we have evaluated the basis functions. They can be evaluated in the same way as for triangles.
    \end{rmk}

\begin{rmk}[Scheme for Polygons]
The extension of \cite{Abgrall2023b} to polygons has been made in \cite{BP_Pampa_VEM}. The extension of what has been done in this section to polygons is quite natural, and will be the topic of a future publication.
\end{rmk}
}

\subsection{Interpretation of the \pampa scheme.}\label{se_interpretation}
Again, we are dealing with the linear advection equation, with a constant speed. We will introduce the flux, $\bbf(u)=\bba\, u$. 
For the sake of simplicity, we reduce ourselves to the case of third order, i.e. DoFs are point values and average, and we discuss only the one dimensional and  triangular cases for simplicity. In the triangular case, again $\lambda_1$, $\lambda_2$ and $\lambda_3$ are the barycentric coordinates. The \dg scheme in one element writes
$$\mathtt{M}_K\dfrac{\mathrm dU_K}{\mathrm dt}+F=0$$ where $U_K=(\{u_{\sigma}\}_{\sigma\in  K}, \bar{u}_K)$ and $F=(\{F_{\sigma}\}_{\sigma\in  K}, \bar{F}_K)$ with 
$$\begin{array}{l}F_{\sigma}=
    -\int_{{K}} \nabla \varphi_{\sigma} \cdot \bbf(\uh)\; \mathrm d\bbx+\int_{\partial K}\varphi_{\sigma}\bbf(\uh)\cdot \bbn \; \mathrm d\gamma,\\ ~ \\\bar{F}_K=-\int_{{K}} \nabla \bar\varphi \cdot \bbf(\uh)\; \mathrm d\bbx+\int_{\partial K}\bar \varphi\bbf(\uh)\cdot \bbn \; \mathrm d\gamma
    \end{array}
$$

The Euler forward is
\begin{subequations}\label{interpretation}
\begin{equation}\label{update}
U_{\vert K}^{n+1}=U_{\vert K}^n-\Delta t M_K^{-1}F_K
\end{equation}
We note that
$$M_K^{-1}=P_K=\frac{1}{|K|} \mathcal{P}$$ where the matrix $\mathcal{P}$ does not depend on $K$, because we are using barycentric coordinate: this is as if working in the reference element.  From \eqref{update}, and using the previous results, we see that the average evolves as
$$\bar{u}_K^{n+1}=\bar{u}_K^n-\frac{\Delta t}{\vert K\vert}\int_{\partial K} \bbf(\uh)\cdot \bbn\; \mathrm d\gamma$$ and the point values as
$$u_{\sigma,K}^{n+1}=u_\sigma^n-\frac{\Delta t }{\vert K\vert}\bigg (\mathcal{P}^{-1}F_K\bigg )_\sigma$$
where, using the previous results, we have
$$\bigg (\mathcal{P}^{-1}F_K\bigg )_\sigma=\int_{K}\nabla\varphi_\sigma\cdot \bbf(\uh)\; \mathrm d\bbx+\int_{\partial K}\varphi_{\sigma}\bbf(\uh)\cdot \bbn\; \mathrm d\gamma= \bba\cdot \nabla \uh(\sigma).$$
{ We have written $u_{\sigma,K}^{n+1}$ to emphasis that there are several values of $u_\sigma$ at time $n+1$. 
At time $t_n$, we have data in $V$, so that we need to project the family $\{u_{\sigma,K}^{n+1}\}$ onto $V$.}
One possibility is then: change nothing for the average, and for the points we do:
    \begin{equation}
        \label{projection}
u_\sigma^{n+1}=
\sum_{K, \sigma\in K}{\omega}_{K,\sigma}u_{\sigma,K}^{n+1}
    \end{equation}
\end{subequations}
with  weight ${\omega}_{\sigma,K}$ that are assumed to be positive and satisfy
$$\sum_{K, \sigma\in K}\omega_{K,\sigma}=\Id$$ to define a projection.
This amount to write 
$$ u_\sigma^{n+{1}}= u_\sigma^n-\Delta t \sum_{K, \sigma\in K} \frac{\omega_{K,\sigma}}{\vert K\vert}\big (\mathcal{P}^{-1}F_{K}\big )_{\sigma}.$$

\bigskip 
The next question is how to choose the weights $\omega_{K,\sigma}$? { Looking back at  \cite{Abgrall_AF} with this interpretation, where the mesh is 
$\{x_j\}_{j\in \Z}$ and the elements are $K_{j+1/2}=[x_j, x_{j+1}]$, the choice amounts to be, for $\sigma=x_j$
$$\omega_{K_{j+1/2},x_j}=\frac{a^+}{a^+ +(-a)^+}, \quad \omega_{K_{j-1/2},x_j}=\frac{(-a)^+}{a^+ +(-a)^+},$$
that is
$$\omega_{K_{j+1/2},x_j}=\text{sign}(a), \quad \omega_{K_{j-1/2},x_j}=\text{sign}(-a).$$

However, other choices are certainly possible. }It is easy to see that in 1D, there is an energy inequality when we project with a simple averaging procedure: assume now that the 
scheme writes
$$\bbu^{n+1}_j=\dfrac{ \Delta_{j+1/2}{u_{j, j+1/2}^{n+1}}+\Delta_{j-1/2}u_{j, j-1/2}^{ n+1}}{\Delta_{j+1/2}+\Delta_{j-1/2}}$$
with $\Delta_{l+1/2}=x_{l+1}-x_l$.

This defines an energy diagonal matrix, i.e. a norm, and one can show that
\begin{proposition}\label{pro_energy_stability}
We have the following energy inequality 
\begin{equation}\begin{split}
\sum_{j}(\Delta_{j+1/2}+\Delta_{j-1/2})(u_j)^2&=
\sum_{j}(\Delta_{j+1/2}+\Delta_{j-1/2})\bigg ( \dfrac{ \Delta_{j+1/2}u_{j, j+1/2}+\Delta_{j-1/2}u_{j, j-1/2}}{\Delta_{j+1/2}+\Delta_{j-1/2}}\bigg )^2\\&
\leq \sum_{K}\Delta_{j+1/2}\big ((u_{j,j+1/2})^2+(u_{j,j-1/2})^2\big )^2
\end{split}\label{projection:dg}
\end{equation}
\end{proposition}
\begin{proof}
    For any $A,B,\alpha,\beta\in \R$ with $\alpha, \beta\geq 0$ and $\alpha+\beta=1$, we have 
$$(\alpha A+\beta B)^2=
\alpha^2 A^2+\beta^2 B^2 +2\alpha\beta AB\leq \alpha^2 A^2+\beta^2 B^2+\alpha\beta (A^2+B^2)=\alpha A^2+\beta B^2$$
then setting 
$$\beta_{j}^+=\dfrac{\Delta_{j+1/2}}{\Delta_{j+1/2}+\Delta_{j-1/2}}, \alpha_{j}^{-}=\dfrac{\Delta_{j-1/2}}{\Delta_{j+1/2}+\Delta_{j-1/2}}, \gamma_{j}=\Delta_{j+1/2}+\Delta_{j-1/2}$$
\begin{equation*}
\begin{split}\sum_j \gamma_j u_j^2&=\sum_{j}\gamma_j\big (\alpha_{j}^-u_{j,j-1/2} + \beta_{j}^{+}u_{j,j+1/2}\big )^2\leq
\sum_j \gamma_j (\alpha_{j}^-) (u_{j,j-1/2})^2+(\beta_{j}^{+})(u_{j,j+1/2})^2)\\
&=\sum_{K=[x_j,x_{j+1}]} \bigg(\gamma_j\beta_{j}^-(u_{j,j+1/2})^2+\gamma_{j+1}\alpha_{j+1}^-(u_{j+1,j+1/2})^2\bigg )
\end{split}
\end{equation*}

Then,
\begin{equation*}
\begin{split}\sum_j \gamma_j u_j^2&=\sum_{j}\gamma_j\big (\alpha_{j}^-u_{j}^- + \beta_{j}^{+}u_{j}^+\big )^2\leq
\sum_j \gamma_j (\alpha_{j}^-) (u_{j}^-)^2+(\beta_{j}^{+})(u_{j}^+)^2)\\
&=\sum_{K} \bigg(\gamma_j\beta_{j}^-(u_j^+)^2+\gamma_{j+1}\alpha_{j+1}^-(u_{j+1}^+)^2\bigg )\\
&\leq \sum_K\Delta_{j+1/2}\bigg ( (u_j^+)^2+(u_{j+1}^+)^2\bigg )
\end{split}
\end{equation*}
because
$$\gamma_j\beta_{j}^+=\Delta_{j+1/2} \frac{\Delta_{j+1/2}}{\Delta_{j-1/2}+\Delta_{j+1/2}}\leq 
\Delta_{j+1/2} $$
and $$
\gamma_j\alpha_{j+1}^-= \Delta_{j+1/2}\dfrac{\Delta_{j+1/2}}{\Delta_{j+3/2}+\Delta_{j+1/2}}
\leq \Delta_{j+1/2}.$$
which is nothing more than \eqref{projection:dg} when $\bbu_j^\pm=\bbu_{j, j\pm 1/2}$, etc
\end{proof}

 This strategy, in 1D, is illustrated on figure \ref{fig:2} where the problem is that of the convection of $\cos(2\pi x)$ with constant speed. The plot shows on a $100$ points grid the result after $10$ and $100$ periods compared to the exact solution. We see very little dispersion. 
 \begin{figure}[h]
 \begin{center}
   \subfigure[]{ \includegraphics[width=0.45\textwidth]{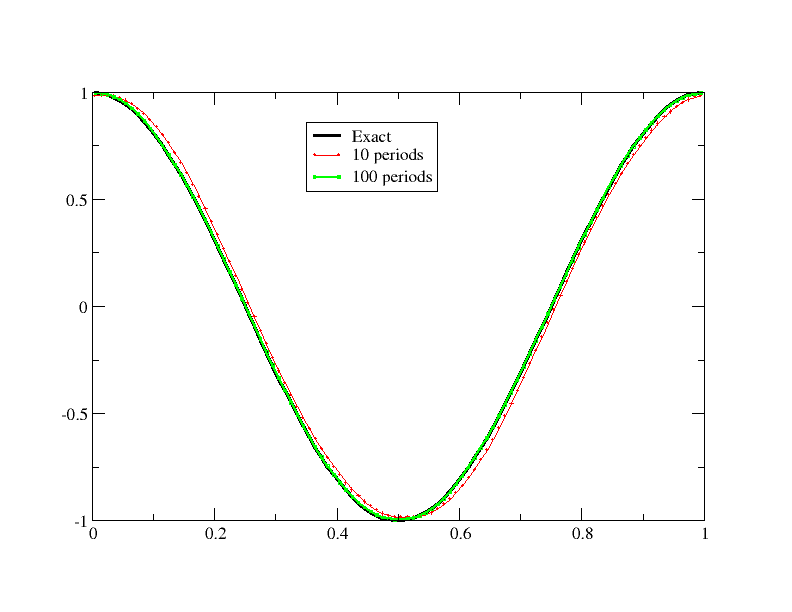}}\subfigure[]{\includegraphics[width=0.45\textwidth]{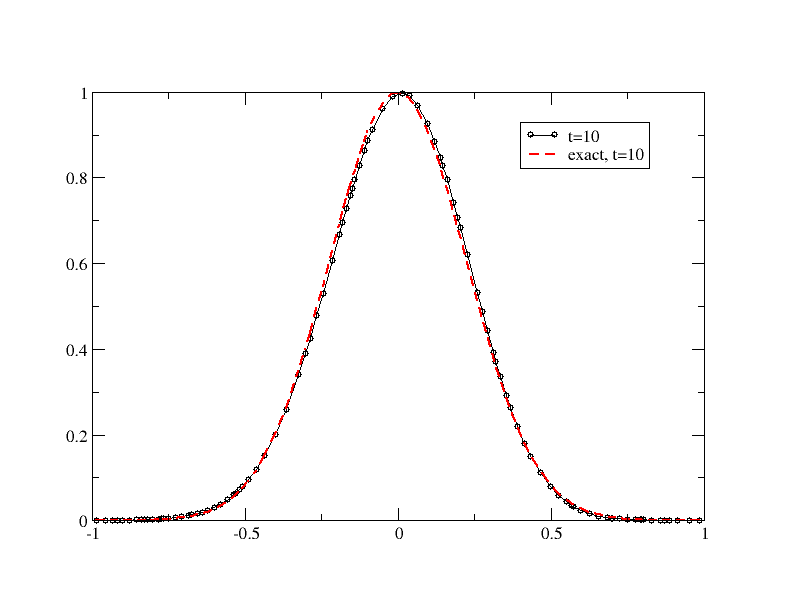}}
 \end{center}
 \caption{\label{fig:2} Solution of $u_t+u_x=0$ with periodic boundary conditions on $[0,1]$ for the initial condition $u_0=\cos(2\pi x)$ after $10$ and $100$ periods, compared to the exact solution, on regular mesh (fig. a). On fig. (b), the initial solution is $u_0=e^{-10x^2}$ on $[-1,1]$ after 10 rotations on a random mesh. }

 \end{figure}
  In two dimensions, the choice \eqref{projection} with weights defining an arithmetic average or an average weight by the area of the elements, as inspired by the 1D case,  seems to lead to unstable results, but the choice
  \begin{equation}
      \label{upwind:weights}
      \tilde{\omega}_{\sigma,K}=\bigg (\sum_{K, \sigma\in K} \big (\text{sign}(\nabla\bbf(\bbu_\sigma)\bullet\bbn_\sigma^K) +\varepsilon\Id_k\big )\bigg)^{-1}\; \bigg ( \text{sign}\big (\nabla\bbf(\bbu_\sigma)\bullet\bbn_\sigma^K\big )+\varepsilon \Id_k\bigg )
  \end{equation}
  where, for $\mathbf{A}=(A,B)$, $A,B\in M_k(\R)$ and $\mathbf{n}=(n_x,n_y)$, $\mathbf{A}\bullet\bbn=An_x+Bn_y$ and $\varepsilon$ is a small positive parameter $\approx 10^{-20}$ to avoid inversion problems\footnote{This can occur for example in the case where $\bbf(\bbu)=\bba\bbu$, $\sigma$ is the mid point of an edge parallel to $\bba$. In that case up-winding makes no sense and the weight $\tilde{\omega}_\bsigma$ is equal to $\tfrac{1}{2}$.}. {This choice was made in  \cite{Abgrall2023b}, and then in   \cite{BP_Pampa_VEM} in the case of polygons.}
  
  {This is illustrated by the following results, where, in $\Omega=[-20,20]$,
  $$\bbu_0(\bbx)=\exp(-\alpha \Vert \bbx-\bbx_0\Vert^2),  \alpha=0.25$$
  on the problem 
  \begin{equation}
  \label{eq:rotation}\dpar{\bbu}{t}+\text{ div }\bbf(\bbu)=0
  \end{equation}
  with
  \begin{subequations}\label{flux}
  \begin{equation}
  \label{flux:translation}
  \bbf(u)=\bba u, \quad \bba(=-1,-1), \quad \bbx_0=(15,15)
  \end{equation}
  and final time $T=30$
  or
  \begin{equation}
  \label{flux:rotation}
    \bbf(\bbx, \bbu)=2\pi (y,-x)u, \bbx=(x,y), \quad \bbx_0=(-10,0)
    \end{equation}\end{subequations} The final time is $T=1$, i.e. one full rotation.

The tables \ref{table:transport} and  \ref{table:rotation} show the error  for the scheme with upwind weights for \eqref{eq:rotation} with the flux \eqref{flux}.  The CFL is set to $0.3$.
\begin{table}[!h]
\begin{center}
\begin{tabular}{|c||cc||cc||cc|}\hline
\multicolumn{7}{|c|}{Average values}\\
\hline
$h$                      &$L^1$              & slope   &$L^2$                   & slope  &$L^\infty$ & slope \\ 
\hline
 $0.4100$ & $0.4412\;10^{-3} $ &    -                   & $0.4412\;10^{-3}$ &    -                   & $ 0.4901\;10^{-1}$&    -                   \\
 $0.2975$ & $0.1958\;10^{-3}$ & $ 2.533$ & $ 0.1224\;10^{-2}$ & $ 2.443$ & $ 0.2313\;10^{-1}$ & $ 2.341$\\
 $0.2442$ & $0.1009\;10^{-3} $ & $3.361$ & $ 0.6420\;10^{-3}$ & $ 3.274$ & $ 0.1233\;10^{-1}$ & $ 3.189$\\
 $0.2083$ & $0.5819\;10^{-4}$ & $ 3.460$ & $ 0.3732\;10^{-3}$ & $ 3.409$ & $ 0.7148\;10^{-2}$ & $ 3.429$\\
 $0.1533$ & $0.2073\;10^{-4}$ & $ 3.364$ & $ 0.1347\;10^{-3}$ & $ 3.321$ & $ 0.2621\;10^{-2}$ & $ 3.270$\\
      \hline \hline \multicolumn{7}{|c|}{Point values}\\
      \hline
      $h$                &$L^1$              & slope     &$L^2$                & slope &$L^\infty$                  & slope \\ 
\hline
 $0.4100$ & $ 0.2915\;10^{-3}$ &    -                   & $ 0.2915\;10^{-3}$ &    -                   & $  0.5018\;10^{-1}$&    -                   \\
 $0.2975$ & $ 0.1295\;10^{-3}$ & $ 2.528$ & $ 0.1006\;10^{-2}$ & $ 2.456$ & $ 0.2338\;10^{-1}$ & $ 2.381$\\
 $0.2442$ & $ 0.6689\;10^{-4}$ & $ 3.351$ & $ 0.5262\;10^{-3}$ & $ 3.285$ & $ 0.1238\;10^{-1}$ & $ 3.222$\\
 $0.2083$ & $ 0.3863\;10^{-4}$ & $ 3.451$ & $ 0.3056\;10^{-3}$ & $ 3.416$ & $ 0.7208\;10^{-2}$ & $ 3.401$\\
 $0.1533$ & $ 0.1379\;10^{-4}$ & $ 3.357$ & $ 0.1102\;10^{-3}$ & $ 3.325$ & $ 0.2633\;10^{-2}$ & $ 3.284$\\
      \hline
      \end{tabular}
      \end{center}
      \caption{\label{table:transport}Errors for the average and point values, triangular mesh, translation problem \eqref{eq:rotation}-\eqref{flux:translation} $T=30$.}
   \end{table}

\begin{table}[!h]
\begin{center}
\begin{tabular}{|c||cc||cc||cc|}\hline
\multicolumn{7}{|c|}{Average values}\\
\hline
$h$                      &$L^1$              & slope   &$L^2$                   & slope  &$L^\infty$ & slope \\ 
\hline
$0.4100$ & $ 0.4575\,10^{-3}$ &      -                   &$ 0.4575\,10^{-3}$     & -                      &$ 0.7373\,10^{-1}$&- \\
 $0.2975$ & $0.3017\,10^{-3 }$ & $1.297$ & $ 0.1883\,10^{-2} $ & $2.403 $ & $0.3568\,10^{-1} $ & $2.262$\\
$ 0.2442$ & $ 0.1551\,10^{-3 }$ & $3.376$ & $ 0.9853\,10^{-3} $ & $3.284$ & $0.1900\,10^{-1}$ & $ 3.197$\\
$ 0.2083$ & $ 0.8885\,10^{-4}$ & $ 3.500 $ & $0.5702\,10^{-3} $ & $3.437$  & $0.1118\,10^{-1} $ & $3.334$\\
$ 0.1533$ & $ 0.3156\,10^{-4}$ & $3.375 $ & $0.2042\,10^{-3} $ & $3.349$  & $0.4074\,10^{-2} $ & $3.290$\\
      \hline \hline \multicolumn{7}{|c|}{Point values}\\
      \hline
      $h$                &$L^1$              & slope     &$L^2$                & slope &$L^\infty$                  & slope \\ 
\hline
$0.4100 $ & $0.6794\,10^{-3} $ &     -                   &$0.6794\,10^{-3}$ &    -                   & $ 0.7538\,10^{-1}$&    -                   \\
 $0.2975 $ & $0.2034\,10^{-3}$ & $ 3.760$ & $ 0.1549\,10^{-2} $ & $2.436 $ & $0.3596\,10^{-1} $ & $2.307$\\
$ 0.2442 $ & $0.1041\,10^{-3 }$ & $3.395 $ & $0.8083\,10^{-3 }$ & $3.298 $ & $0.1922\,10^{-1} $ & $3.178$\\
$ 0.2083$ & $ 0.5968\,10^{-4 }$ & $3.498 $ & $0.4671\,10^{-3 }$ & $3.446 $ & $0.1125\,10^{-1}$ & $ 3.363$\\
$ 0.1533$ & $ 0.2115\,10^{-4 }$ & $3.382 $ & $0.2042\,10^{-3} $ & $2.699 $ & $0.4084\,10^{-2} $ & $3.305$\\
      \hline
      \end{tabular}
      \end{center}
      \caption{\label{table:rotation}Errors for the average and point values, triangular mesh, rotation problem \eqref{eq:rotation}-\eqref{flux:rotation}, $T=1$.}
   \end{table}
This shows that the scheme delivers the expected error, even a little bit more. This is likely be a coincidence.

}

 \subsection{Generatisation to non constant advection speeds and non linear problems.}
{In \cite{Abgrall2023b} and \cite{BP_Pampa_VEM}, the update of the average is done  by
$$\vert K\vert \dfrac{\mathrm d\bar u_K}{\mathrm dt}+\int_{\partial K}\bbf(u)\cdot \bbn \; \mathrm d\gamma=0,$$
while the update of the point values is done by
$$\dfrac{\mathrm d u_{\sigma}}{\mathrm dt}+\sum_{K, \sigma\in K}\omega_{\bsigma, K} \mathbf{J}(u_\bsigma)\nabla u_{\vert K}(\bsigma)=0$$
where $\mathbf{J}$ is the Jacobian of the flux with respect to the conservative variables and 
$\omega_{\bsigma, K} $ is obtained by a generalisation of \eqref{upwind:weights}, see \cite{Abgrall2023b,BP_Pampa_VEM} for details.
The main difference with the linear convection problem with constant speed is that $\text{div }\bbf(u)$ is not an element of $V$, so that the equivalence is lost.
Nevertheless, one can extend what we have already written in that case, it will be a different scheme that needs to be studied. This interpretation of the \pampa scheme allows in particular to understand how to discretise boundary conditions, see \cite{PampaDG} for such a study that is above the present contribution. }

\section{Intrinsic positivity properties of \pampa}
{In all the examples described bellow, as well as in \cite{BP_Pampa_VEM}, the solution of \eqref{eq:hyperbolic} is approximated by
$$\bbu_{\mathrm h}=\sum_{\sigma\in \partial K} \bbu_\sigma \varphi_\sigma+\bar \bbu_K \bar \varphi_K$$
where
\begin{itemize}
\item $\bar\varphi_K=0$ on $\partial K$, and $\int_K \bar\varphi_K\;\mathrm d\bbx=|K|$,
\item $\varphi_\sigma(\sigma')=\delta_\sigma^{\sigma'}$, and for all $\sigma$,
$$\int_K\varphi_\sigma\; \mathrm d\bbx=0$$
\end{itemize}
We will call $\bar\varphi$ a bubble function, and in each cases, it admits a maximum that we will denote by $\bbx^\star$. This maximum is in the interior of the polygon. }

In this section, we show the following result:
\begin{proposition}\label{prop}
    Let $K$ be an element or a polygon. For the 1D third order scheme, the 2D quadratic, cubic, cubic moment and the scheme of \cite{BP_Pampa_VEM}, we have the following property: There exists $c_0>0$ depending only on $K$ with the following property: If $\bbu^n_\sigma\in \mathcal{D}$ and $\bbu(\bbx^\star)\in \mathcal{D}$ then 
    $$\bar\bbu^{n+1}_K=\bar\bbu^n_K-\frac{\Delta t}{\vert K\vert}\oint_{\partial K}\bbf(\bbu_{\mathrm h})\cdot \bbn\; \mathrm d\gamma$$ satisfies $\bar\bbu^{n+1}_K\in \mathcal{D}$ if
    $$\Delta t\max_{\bbx\in K}\rho(\nabla \bbf(\bbu_{\mathrm h}(\bbx)))\leq c_0\frac{\vert K\vert}{\vert\partial K\vert}.$$
\end{proposition}
In order to establish this property, we proceed in 2 steps. First we show the following lemma
\begin{lemma}\label{lemma:simpson}
    For the functional approximations mentioned in proposition \ref{prop},  there exists $\bbx^\star\in \text{interior of } K$ such that
$$\bbu_{\mathrm h}(\bbx^\star)=\sum_{\sigma\in \partial K} \alpha_\sigma^K \bbu_\sigma+\omega_K\bar u_K$$ with 
$$\omega_K>0 \text{ and } \alpha_\sigma^K<0.$$
\end{lemma}
In a second step, we use this property, which can be seen as a generalization of Simpson's formula, to established proposition \ref{prop}.

{In general we will have $\bbu(\bbx^\star), \bbu_\sigma\in \mathcal{D}$ will imply that $\bar u_K\in \mathcal{D}$. But the converse is wrong. For example
$u=6x(1-x)+7x(3x-2)$ is such that $u(0)=0$, $u(1)=7$ and $\bar u=1$, but $u(1/2)=-1/4<0$. This means that when $u(\bbx^\star)\not\in \mathcal D$, something else must be done. One possibility is described in \cite{BP_Pampa_VEM} where a monolitic convex limiting is used. Another possibility, getting inspiration from \dg is to "limit" $\bbu$ in $K$ while keeping the average. Such  a solution is described in \cite{Pampa1D}, the price to pay is to use numerical flux to update the average value. What we conjecture is that a result of the type described in Lemma \ref{lemma:simpson} is always true. 
}

\subsection{The 1D case.}
We begin with the 1D case, repeating \cite{Pampa1D}. The computational domain is covered by a set of non-overlapping cells denoted by $I_\jph=[x_j,x_{j+1}]$ centered at $x_\jph$. The solution given in terms of boundary DoFs $u_j^n\approx u(x_j,t^n)$ and internal DoFs $\xbar u_\jph\approx\int_{I_\jph}u(x)\;\mathrm dx/\dx_{j+1/2}$ are assumed to be available. For each cell $I_\jph$, the internal DoF is evolved with
\begin{equation}\label{eq:IDoF}
    \xbar u_{j+1/2}^{n+1}=\xbar u_{j+1/2}^n-\lambda\big (f(u_{j+1})-f(u_j)\big )
\end{equation}
where $\lambda=\frac{\dt}{\dx_{j+1/2}}$ with $\dx$ being the measure of the cell $I_\jph$ and $\dt$ the adaptive time step depends on a certain CFL condition. We consider the third-order of accuracy here and thus can apply the Simpson's rule to write
\begin{equation*}
    \xbar u_{j+1/2}^n=\frac{1}{6}\big ( u_j^n+4 u_{j+1/2}^n+u_{j+1}^n\big ),
\end{equation*}
so that \eqref{eq:IDoF} becomes
\begin{equation*}
\begin{split}
\xbar u_{j+1/2}^{n+1}&=\frac{1}{6}\bigg (u_{j+1}^n-6\lambda\big (f(u_{j+1})-\widehat f(u_{j+1}, u_{j+1/2})\big )\bigg )\\
&+\frac{4}{6}\bigg ( {u}_{j+1/2}^n-\frac{6}{4}\lambda\big (\widehat f(u_{j+1},{u}_{j+1/2})-\widehat f({u}_{j+1/2}, u_j)\big )\bigg )\\
&+\frac{1}{6}\bigg (u_{j}^n-6\lambda\big (\widehat f({u}_{j+1/2},u_j)-f(u_j)\big )\bigg )\\
\end{split}
\end{equation*}
Then, we look at the intervals:
in the cell $[x_{j-1}, x_j]$ we have 3 constant cells with value $u_{j-1}$, $ u_{j-1/2}$, $u_j$ and in the cell $[x_j,x_{j+1]}]$, we have 3 constant cells with values
$u_j$, $u_{j+1/2}$, $u_{j+1}$, so that we can interpret the term $f(u_j)$ as the flux $\widehat f(u_{j},u_j)$ between the most right subcell of $[x_{j-1},x_{j}]$ and the most left subcell of $[x_j,x_{j+1}]$. This shows that if $6\lambda\leq c_0$ the CFL constant for the monotone flux $\widehat f$, if $u_{j+1/2}$ is in the bounds $[m,M]$ as well as the other terms, then $\bar u_{j+1/2}^{n+1}\in [m, M]$.
{ In the one dimensional cubic case, the bubble function is $6x(1-x)$ and it reaches its maximum for $x=\tfrac{1}{2}$.}

The key points are 
\begin{itemize}
\item The quadrature formula with positive weights,
\item The continuity of the approximation at the boundaries of the cell,
\end{itemize}

\subsection{The 2D case.}
If we can find a point such that
$$u(\bbx^\star)=\sum_{\text{Lagrange points}}\omega_{\sigma} u(\bbx_\sigma)+\omega_K \bar u_K$$
with $\omega_K>0$ and $\omega_{\sigma}<0$, we can repeat the same argument.

\subsubsection{Polynomial case: quadratic and cubic cases of \cite{Abgrall2023b}.}

For quadratic approximation, we have
$$u_{\mathrm h}(\bbx)=\sum_\sigma u(\bbx_\sigma)\varphi_\sigma(\bbx)+\bar u_K \bar\varphi_K(\bbx)$$
with
$$\bar\varphi_K=60\lambda_1\lambda_2\lambda_3$$
$$\varphi_{\sigma_i}=(2\lambda_i-1)\lambda_i, \quad i=1,2,3$$
and 
$$\varphi_{\sigma_4}=4\lambda_1\lambda_2-\frac{1}{3}\bar\varphi_K,\quad \varphi_{\sigma_5}=4\lambda_2\lambda_3-\frac{1}{3}\bar\varphi_K,\quad \varphi_{\sigma_6}=4\lambda_3\lambda_1-\frac{1}{3}\bar\varphi_K,$$
and we see that for the centroid $\bbx_K$,
$$\bar\varphi_K(\bbx_K)=\frac{60}{27}, \quad \varphi_{\sigma_i}=-\frac{1}{9},\quad i=1,2,3; \quad \varphi_{\sigma_i}(\bbx_K)=\frac{4}{9}-\frac{20}{27}=-\frac{8}{27},\quad i=4,5,6$$
i.e.
$$u(\bbx_K)=\frac{20}{9}\bar u_K-\frac{1}{9}\sum_{i=1}^3 u_{\sigma_i}-\frac{8}{27}\sum_{i=4}^6u_{\sigma_i}$$
from which we get
$$\bar u_K=\frac{9}{20}u(\bbx_K)+\frac{1}{20}\sum_{i=1}^3 u_{\sigma_i}+\frac{2}{15}\sum_{i=4}^6u_{\sigma_i}.$$

\medskip

For cubic approximation, we have 
$$\varphi_{\sigma_i}=\frac{1}{2}\lambda_i(3\lambda_i-1)(3\lambda_i-2)-\frac{60}{30}\lambda_1\lambda_2\lambda_3,\quad  i=1,2,3$$
and
\begin{equation*}
    \begin{split}
       \varphi_{\sigma_4}=\frac{9}{2}\lambda_1\lambda_2(3\lambda_1-1)-\frac{9}{2}\lambda_1\lambda_2\lambda_3, \quad \varphi_{\sigma_5}=\frac{9}{2}\lambda_1\lambda_2(3\lambda_2-1)-\frac{9}{2}\lambda_1\lambda_2\lambda_3,\\
       \varphi_{\sigma_6}=\frac{9}{2}\lambda_2\lambda_3(3\lambda_2-1)-\frac{9}{2}\lambda_1\lambda_2\lambda_3, \quad \varphi_{\sigma_7}=\frac{9}{2}\lambda_2\lambda_3(3\lambda_3-1)-\frac{9}{2}\lambda_1\lambda_2\lambda_3,\\
    \varphi_{\sigma_8}=\frac{9}{2}\lambda_3\lambda_1(3\lambda_3-1)-\frac{9}{2}\lambda_1\lambda_2\lambda_3, \quad \varphi_{\sigma_9}=\frac{9}{2}\lambda_3\lambda_1(3\lambda_1-1)-\frac{9}{2}\lambda_1\lambda_2\lambda_3,
    \end{split}
\end{equation*}
and we see that
$$u(\bbx_K)=-\sum_{j=1}^3 \frac{2}{27}u(\sigma_i)-\sum_{i=4}^9\frac{1}{6}u(\sigma_i)+\frac{60}{27}\bar u_K.$$
From this we get
$$\bar u_K=\frac{9}{20}u(\bbx_K)+\frac{1}{30}\sum_{j=1}^3 u(\sigma_i)+\frac{9}{120}
\sum_{i=4}^9u(\sigma_i).$$

We note that
$$\frac{9}{20}+\frac{3}{20}+\frac{6}{15}=1$$
and 
$$\frac{9}{20}+\frac{3}{30}+\frac{54}{120}=1.$$
\subsubsection{Polynomial case: cubic moment.}
    We write
    $$u_{\mathrm h}(\bbx)=\sum_{\bsigma\in \partial K}u_\bsigma\varphi_\bsigma(\bbx)+\sum_{\bmu}m_\bmu(u)\bar\varphi_\bmu(\bbx)$$
    and we evaluate this at the centroid.
    Because the coefficients defining the $\bar\varphi_\bmu$ are obtained by cyclic permutation, from lemma \ref{lemma:dual_basis}, we see that
    $$\varphi_\bmu(\frac{1}{3},\frac{1}{3}, \frac{1}{3})=\frac{1}{3^P}\sum_{\bnu} a_{\bmu+\bnu}$$ where $P$ only depends on the degree ($P=3+(k-2)$). This relation shows that 
    $$\omega_K=\varphi_\bmu(\frac{1}{3},\frac{1}{3}, \frac{1}{3})$$ does not depend on $\bmu$ and, because the sum of the moments is the average,
    we obtain
    $$u_{\mathrm h}\left(\frac{1}{3},\frac{1}{3}, \frac{1}{3}\right)=\sum_\bsigma u_\bsigma\varphi_\bsigma\left(\frac{1}{3},\frac{1}{3}, \frac{1}{3}\right)+\omega_K \xbar \bbu_K$$
    The only thing to check is if $\omega_K>0$ and 
    $$\varphi_\bsigma \left(\frac{1}{3},\frac{1}{3}, \frac{1}{3} \right)<0.$$

    In the cubic case $(k=3)$, $P=4$ and 
    $$\omega_K=\frac{1}{3^4}\big ( \frac{1800-720-720}{7}\big )=\frac{360}{7\times 3^4}=\frac{360}{567}>0$$
    and $\varphi_{\bsigma}(\frac{1}{3},\frac{1}{3}, \frac{1}{3})=P_\bsigma(\frac{1}{3},\frac{1}{3}, \frac{1}{3})-\sum_\bmu\frac{m_\bmu(P_\bsigma)}{m_\bmu(\varphi_\bmu)}\varphi_\bmu(\frac{1}{3},\frac{1}{3},\frac{1}{3})$.
    The Gauss--Lobatto points in $[0,1]$ are $\{\alpha_0=0,\alpha_1=\tfrac{\sqrt{5}-1}{2\sqrt{5}}, \alpha_2=\tfrac{\sqrt{5}+1}{2\sqrt{5}}, \alpha_3=1\}$. The Lagrange polynomials are (the interpolation points are given by their barycentric coordinates)
    \begin{itemize}
        \item for $(1,0,0)$, 
        $$P_0=\dfrac{\lambda_1(\lambda_1-\alpha_1)(\lambda_1-\alpha_2)}{(1-\alpha_1)(1-\alpha_2)}$$ and its value at the centroid is
        $$-\frac{1}{27}<0.$$
        \item For $(\alpha_2,\alpha_1,0)$, it is 
        $$P_1=\frac{\lambda_1\lambda_2(\lambda_1-\alpha_1)}{\alpha_2\alpha_1(\alpha_2-\alpha_1)}$$
        and its value at the centroid is 
        $$\frac{5}{18}-\frac{5\sqrt{5}}{54}>0$$
        \item for $(\alpha_1,\alpha_2,0)$ it is
        $$P_2=\frac{\lambda_1\lambda_2(\lambda_1-\alpha_2)}{\alpha_1\alpha_2(\alpha_1-\alpha_2)}$$ and its value at the centroid is
        $$\frac{5}{18} + \frac{5\sqrt{5}}{54}>0$$
    \item for $(0,1,0)$, 
        $$P_3=\dfrac{\lambda_2(\lambda_2-\alpha_1)(\lambda_2-\alpha_2)}{(1-\alpha_1)(1-\alpha_2)}$$ and its value at the centroid is
        $$-\frac{1}{27}<0.$$    
    \end{itemize}
    Next, we know the value $\varphi_\bmu(\frac{1}{3},\frac{1}{3},\frac{1}{3})=\frac{360}{567}$ from lemma \ref{lemma:dual_basis} and the form of $\varphi_{\bmu}$. Moreover, we can compute $m_\bmu(\varphi_\bmu)$ and $m_\bmu(P_\bsigma)$. Finally, $\varphi_\bsigma(\frac{1}{3},\frac{1}{3},\frac{1}{3})<0$ is verified.
    So again we have the positivity.
\subsubsection{Approximation using Virtual finite element (VEM) approximation}
{
In \cite{BP_Pampa_VEM}, \remi{we have  }extended the method developed in \cite{Abgrall2023b}. The computational domain is covered by a family of non overlapping polygon denoted by $K$, $\Omega=\cup K$.
In each polygon, the solution is approximated by an element of 
$$V_k(K)=\{ v\text{ such that } v_{\mid\partial K}\in \P^k(\partial K) \text{ and }\Delta v\in \P^{k-2}(K)\},$$
where $k\geq 2$. In $\Omega$, the solution will be approximated in 
$$V_k(\Omega)=\bigg(\bigoplus_K V_k(K)\bigg )\cap C^0(\Omega).$$

This kind of approximation was introduced in \cite{brezzi,hitch,vem} in a variational framework. This is an extension of the classical finite element techniques, where the ``elements'' are no longer simplex but general polygons with very mild assumptions on the polygons (essentially that they are star shaped with respect to one point). Hence there is no longer a reference element, so that basis functions must be designed for each polygons. This is theoretically possible, at least analytically, but very cumbersome. Hence the idea behind VEM is to avoid to explicitly use basis functions. We sketch the framework, some more details are given in the appendix, the interested reader is suggested to study \cite{brezzi} for a review.
The degrees of freedom are the Gauss-Lobato points on each edge of $K$ and the moments
$$m_\bmu(u)=\frac{1}{\vert K\vert}\int_K\bigg (\frac{\bbx-\bbx^\star}{h_K}\bigg )^\bmu u(\bbx)\; \mathrm d\bbx, \qquad \vert \bmu\vert \leq k-2$$ where we have introduced the following notations:
\begin{itemize}
    \item $\bbx^\star$ is a point toward which $K$ is star-shaped,
    \item $\bmu=(\mu_1, \ldots, \mu_d)$ is a multi-index, $\vert \bmu\vert=\sum_{i=1}^d\mu_i$.
    \item If $\bby=(y_1, \ldots , y_d)$,
    $$y^\bmu=\Pi_{i=1}^d y_i^{\mu_i}.$$
    \item $\sigma$ is any of the Gauss-Lobatto points.
    \item In the following $\pi$ is the $L^2$ projector defined on $V_k(K)$ onto $\P^k(K)$, see \cite{hitch,vem}. It is computable solely with the given DoFs, and there is no need to know the basis functions.
    \item We use the following notations for the ``basis'' functions: for each DoF on the boundary, we call $\varphi_\sigma$ the element of $V_k(K)$ such that 
    $$\varphi_\sigma(\sigma')=\delta_\sigma^{\sigma'}, m_\bmu(\varphi_\sigma)=0, \quad \forall \bmu, \vert\bmu\vert\leq k-2$$ and by $\varphi_\bmu$  the element of $V_k(K)$ such that
    $$\forall \sigma, \varphi_\bmu(\sigma)=0\text{ and }
    m_{\bmu'}(\varphi_\bmu)=\delta_{\bmu}^{\bmu'}.$$
\end{itemize}

}
The question we want to address is that of the existence of generalised Simpson rules for polygons which are not assumed to \remi{ be }convex but assumed to be star shaped with respect to one point in $K$.
This case is a bit more involved, and we consider the quadratic case only.
Again we have on the polygon $K$
$$\uh(\bbx)=\sum_{\sigma\in \partial K}u(\bbx_\sigma)\varphi_\sigma(\bbx_\sigma)+\bar u_K\bar\varphi_K(\bbx)$$
where the ``basis'' functions $\varphi$ are such that:
\begin{itemize}
\item $\Delta \varphi_\sigma=\alpha_\sigma\in \R$, $\Delta  \bar\varphi_K=\alpha_K\in \R$,
\item $\bar\varphi_K=0$ on $\partial K$ and $\int_K\bar\varphi_K\;\mathrm d\bbx=\vert K\vert$,
\item $\varphi_\sigma(\sigma')=\delta_\sigma^{\sigma'}$, $\varphi_\sigma\in \P^2(\partial K)$ and $\int_K\varphi_\sigma\;\mathrm d\bbx=0$.
\end{itemize}
The idea is to find a point $\bbx^\star\in \overset{\circ}{\Omega}$  such that 
$$u(\bbx^\star)=\sum_{\sigma\in \partial \Omega} u_\sigma \varphi_\sigma(\bbx^\star)+\bar u\bar\varphi_K(\bbx^\star)$$ where
$$\bar\varphi_K(\bbx^\star)>0\text{ and } \bar\varphi_\sigma(\bbx^\star)<0 \quad \forall \sigma\in \partial K.$$

\bigskip

It is easy to show that $\bar\varphi_K\geq 0$ on $K$. First $\alpha_K<0$ because
$$\alpha_K\vert K\vert=\int_K\bar\varphi_K\Delta \bar\varphi_K\;\mathrm d\bbx=-\int_K \nabla \bar\varphi_K^2\;\mathrm d\bbx+\int_{\partial K}\bar\varphi_K\nabla\bar\varphi_K\cdot \bbn\;\mathrm d\gamma=-\int_K \nabla \bar\varphi_K^2\;\mathrm d\bbx<0$$
and then $\Delta \bar\varphi_K=\alpha_K<0$, so that the maximum principle shows that $\varphi\geq 0$ on $K$. If $\bbx$ is in the interior of $K$, the same maximum principle (more precisely the mean value theorem) shows that $\bar\varphi_K(\bbx)>0$.

We call $\bbx^\star$ a point for which $\max\limits_{\bbx\in \Omega}\bar\varphi_K(\bbx)=\bar\varphi_K(\bbx^\star)$. We have $\bbx^\star\in \overset{\circ}{\Omega}$ and $\bar\varphi_K(\bbx^\star)>1$ because, since
$$\vert K\vert \; \bar\varphi_K(\bbx^\star)\geq \int_K\bar\varphi_K(\bbx)\; \mathrm d\bbx=\vert K\vert,$$
we see that $\bar\varphi_K(\bbx^\star)\geq 1$ and if $\bar\varphi_K(\bbx^\star)= 1$, we would have
$$0=\int_K \bar\varphi_K(\bbx)\; \mathrm d\bbx-\vert K\vert =\int_K\big ( \bar\varphi_K(\bbx)-1\big )\; \mathrm d\bbx$$
with $\bar\varphi_K(\bbx)-1\leq 0$, so $\bar\varphi_K(\bbx)-1=0$ and this is not possible.

Next we show that  $\alpha_\sigma>0$. Let $\bby_\sigma$ a point where $\varphi_\sigma$ reaches its minimum. It  must be such that $\varphi_\sigma(\bby_\sigma)<0$ because 
$$\int_K\varphi_\sigma(\bbx)\; \mathrm d\bbx=0.$$

We have two cases to look at:
 \begin{itemize}
\item If $\sigma$ is a midpoint, then $\varphi_\sigma\geq 0$ on the boundary. If $\alpha_\sigma\leq 0$, then from the maximum principle, $\varphi_\sigma\geq 0$ on $\Omega$ and then we cannot have 
$$\int_K\varphi_\sigma(\bbx)\; \mathrm d\bbx=0.$$
So $\alpha_\sigma>0$.
\item If $\sigma$ is a vertex: Assume that $\alpha_\sigma<0$.

If $\Delta \varphi_\sigma=\alpha_\sigma<0$, then we take $u=\varphi_\sigma+\theta\bar\varphi_K$ with $\theta>0$.
If the minimum of $\varphi_\sigma$ is reached on the boundary, call it $\bby^\star$:
$\min_\Omega\varphi_\sigma=\min_{\partial \Omega}\varphi_\sigma=\varphi_\sigma(\bby^\star)$. We have, because $\bar\varphi_K\geq 0$, $u\geq \varphi_\sigma$, so $\min_{\Omega} u\geq \min_\Omega \varphi_\sigma$. If the minimum of $u$ is reached on $\partial\Omega$
\begin{equation*}
    \begin{split}
\varphi_\sigma(\bby^\star)=\min_\Omega \varphi_\sigma &
=\min_{\partial\Omega} \varphi_\sigma=\min_{\partial\Omega }(\varphi_\sigma+\theta\bar\varphi_K)\geq \min_\Omega(\varphi_\sigma+\theta\bar\varphi_K)\\&\qquad = \varphi_\sigma(\bby^\star_\theta)+\theta\bar\varphi_K(\bby^\star_\theta)
 >\varphi_\sigma(\bby^\star_\theta),
\end{split}
\end{equation*}
 so this is absurd.
Hence, we must have $\bby^\star\not\in \partial\Omega$, so since it is a strict minimum, $\Delta \varphi_\sigma(\bby^\star)=\alpha_\sigma>0$.
\end{itemize}

\medskip
We want to show that $\varphi_\sigma(\bbx^\star)<0$.

We take $\varepsilon>0$ and consider $\bar\varphi_\varepsilon=\log(\bar\varphi_K+\varepsilon)$. It is well defined, and we have
$$\nabla\bar\varphi_\varepsilon=\frac{\nabla\bar\varphi_K}{\bar\varphi_K+\varepsilon}, \nabla^2\bar\varphi_\varepsilon=-\frac{\nabla\bar\varphi_K\otimes\nabla\bar\varphi_K}{(\bar\varphi_K+\varepsilon)^2}+\frac{\nabla^2\bar\varphi_K}{\bar\varphi_K+\varepsilon}, $$ so that
$$\Delta \bar\varphi_\varepsilon=-\frac{\Vert \nabla\bar\varphi_K\Vert^2}{(\bar\varphi_K+\varepsilon)^2}+\frac{\Delta \bar\varphi_K}{\bar\varphi_K+\varepsilon}\leq 0.$$
Then we consider $u=\varphi_\sigma+\theta_\varepsilon\bar\varphi_\varepsilon$ for $\theta_\varepsilon\geq 0$.
It is clear that for $\varepsilon$ small enough, $u\leq 0$ on $\partial\Omega$ if $1+\theta_\varepsilon\log \varepsilon\leq 0$
because $\varphi_\sigma\leq \max\limits_{\bbx\in \Omega}\varphi_\sigma=1$. So we look for $\epsilon_0$ such that if $\epsilon\leq \epsilon_0$,
$$\Delta\varphi_\sigma+\theta_\varepsilon\Delta\bar\varphi_\varepsilon\leq 0$$
We have
$$
\Delta\varphi_\sigma+\theta_\varepsilon\Delta\bar\varphi_\varepsilon=\underbrace{\alpha_\sigma}_{\geq 0}-\theta_\varepsilon \bigg (\underbrace{\frac{\Vert \nabla\bar\varphi_K\Vert^2}{(\bar\varphi_K+\varepsilon)^2}+\frac{\vert\alpha_P\vert}{\bar\varphi_K+\varepsilon}}_{\geq 0}\bigg)\leq 0$$
if 

$$\theta_\varepsilon\bigg ( \frac{\Vert \nabla\bar\varphi_K\Vert^2}{(\bar\varphi_K+\varepsilon)^2}+\frac{\vert\alpha_P\vert}{\bar\varphi_K+\varepsilon}\bigg )\geq {\alpha_\sigma}$$
so we need
$$\theta_\varepsilon\bigg ( \frac{\min_\Omega\Vert \nabla\bar\varphi_K\Vert^2}{(\bar\varphi_K(\bbx^\star)+\varepsilon)^2}+\frac{\vert\alpha_P\vert}{\bar\varphi_K(\bbx^\star)+\varepsilon}\bigg )\geq {\alpha_\sigma}$$
Together with the condition on the boundary, we need:
so we need
$$
\theta_\varepsilon\bigg ( \dfrac{\min_\Omega\Vert \nabla\bar\varphi_K\Vert^2}{(\bar\varphi_K(\bbx^\star)+\varepsilon)^2}+\dfrac{\vert\alpha_P\vert}{\bar\varphi_K(\bbx^\star)+\varepsilon}\bigg )\geq {\alpha_\sigma} \text{ and }
\theta_\varepsilon\log\varepsilon\leq -1
$$
i.e
\begin{equation}\label{eq:1}
\theta_\varepsilon\bigg ( \dfrac{\min_\Omega\Vert \nabla\bar\varphi_K\Vert^2}{(\bar\varphi_K(\bbx^\star)+\varepsilon)^2}+\dfrac{\vert\alpha_P\vert}{\bar\varphi_K(\bbx^\star)+\varepsilon}\bigg )\geq {\alpha_\sigma} \text{ and }
\theta_\varepsilon\log\big (\frac{1}{\varepsilon}\big )\geq 1
\end{equation}

Last we want that
$$\varphi_\sigma(\bbx^\star)\leq -\theta_\varepsilon\log\big(\bar\varphi_K(\bbx^\star)+\varepsilon)$$
with 
\begin{equation}\label{eq:2}\bar\varphi_K(\bbx^\star)+\varepsilon>1
\end{equation} so with $\varepsilon$ small enough.
So we first choose $\varepsilon$ to meet \eqref{eq:2} and then we choose $\theta>0$ so that \eqref{eq:1} is met. This shows that $\varphi_\sigma(\bbx^\star)<0$.

\bigskip
All in all, we have
$$u(\bbx^\star)=\omega_K\bar u_K+\sum_{\sigma\in \partial K}\omega_i u_{\sigma_i},$$
so that
$$\bar u_K=\frac{1}{\omega_K}u(\bbx^\star)-\sum_{\sigma\in \partial K}\frac{\omega_i}{\omega_K} u_{\sigma_i}.$$
Since  (take $u\equiv 1$)
$$1=\omega_K+\sum_{\sigma\in \partial K}\omega_i, $$
we see that
$$\frac{1}{\omega_K}-\sum_{\sigma\in \partial K}\frac{\omega_i}{\omega_K} =\frac{1}{\omega_K}\big (1-\sum_{\sigma\in \partial K}\omega_i\big )=\frac{\omega_K}{\omega_K}=1.$$
\subsection{Bound preserving property}
We assume that \eqref{eq:1} is such that if $u_0(\bbx)\in \mathcal{D}$ for all $\bbx\in \R^d$ (or almost everywhere), then $u(\bbx,t)\in \DD$ for all $\bbx\in \R^d,t>0$.

In what follows, we assume
$$\bar u_K=\alpha_K u(\bbx_K)+\sum_{i=1}^N \alpha_i u(\sigma_i)$$
with $\alpha_K, \alpha_i>0$. We note that 
$$\alpha_K+\sum_{i=1}^N \alpha_i=1 ,$$

We write ($\lambda=\tfrac{\Delta t}{\vert K\vert}$)
\begin{equation*}
\begin{split}
\bar u_K^{n+1}&=\bar u_K^n-\lambda\sum_{i=1}^N \bbf(u_{\sigma_i})\cdot \bbn_i\\
&=\alpha_K u(\bbx_K)+\sum_{i=1}^N \alpha_i u(\sigma_i)-\lambda\sum_{i=1}^N \bbf(u_{\sigma_i})\cdot \bbn_i\\
&=\alpha_K\bigg( u(\bbx_K)-\frac{\lambda}{\alpha_K} \sum_{i=1}^N  \hbbf_{\bbn_i}\big (u_{\sigma_i}, u(\bbx_K)\big )\bigg )\\
&\qquad+\sum_{i=1}^N \alpha_i\bigg ( u(\sigma_i)-\frac{\lambda}{\alpha_i} \bigg [ \hbbf_{\bbn_i}(u_{\sigma_i},u_{\sigma_i})\cdot \bbn_i-\hbbf_{\bbn_i}(u_{\sigma_i}, u(\bbx_K))\bigg ]\bigg )
\end{split}
\end{equation*}
where $\hbbf_\bbn(\uparrow,\downarrow)$. Hence, if $$\lambda\min\big (\alpha_K, \min\limits_i\alpha_i\big )\big )\leq \lambda_0$$
where $\lambda_0$ is the maximum stability parameter for $\hbbf_\bbn$, \remi{we get a convex decomposition}. This shows that under this condition, if $\hbbf_\bbn$ is invariant domain preserving, if $\{u^n_\sigma\in \DD\}$ and $\bar u^n_K\in \DD$ then $\bar u_K^{n+1}\in \DD$.

We note that there is no need to define any control volume, this is a purely algebraic property.

\subsection{Numerical evidence}
In order to illustrate this property, we have used the scheme of \cite{BP_Pampa_VEM} on
$$\dpar{u}{t}+\dpar{u}{x}=0$$ on Jiang-Shu problem with periodic boundary conditions, $300$ mesh points, until $T=2$ with $CFL=0.15<\tfrac{1}{6}$. The results are displayed in figure \ref{figure:positive:scal}. In the original scheme of \cite{BP_Pampa_VEM}, the bound preserving strategy is applied both on the point values and the average values. We have tested the results above where we apply the BP strategy only on the point values.
\begin{figure}[]
\begin{center}
    \subfigure[]{\includegraphics[width=0.75\textwidth]{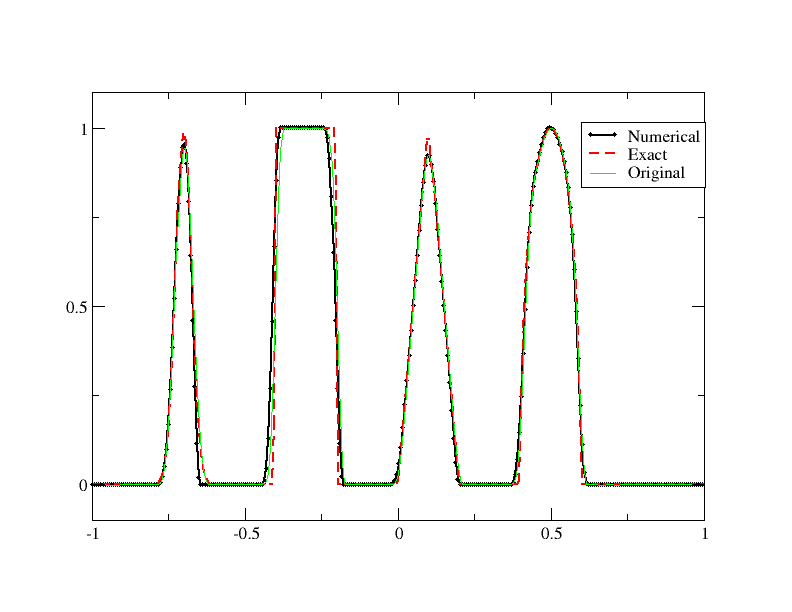}}
     \subfigure[]{\includegraphics[width=0.75\textwidth]{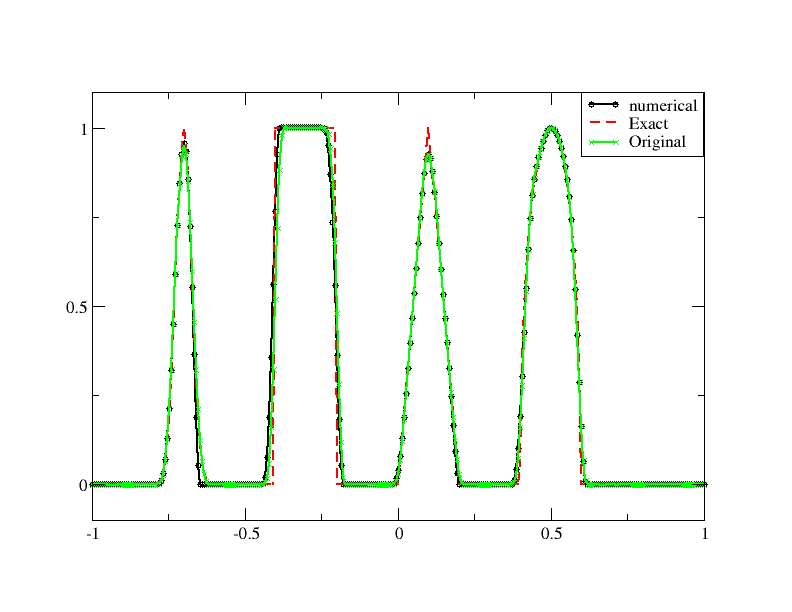}}
     \end{center}
     \caption{\label{figure:positive:scal} Jiang and Shu's problem. (a): average values, (b): point values. Original: scheme of \cite{BP_Pampa_VEM}, Numerical: the BP strategy is applied only on the point values.}
\end{figure}

\remi{We also consider the compressible Euler equations and show the result of the simulation }
Configuration 3 of Lax \& Liu in \cite{zbMATH01929393}, the initial condition is 
$$ (\rho, u,v,p)=\left \{\begin{array}{ll}
(\rho_1,u_1,v_1,p_1)=(1.5, 0, 0, 1.5)& \text{ if } x\geq1\text{ and } y\geq 1,\\
(\rho_2,u_2,v_2,p_2)=(0.5323, 1.206, 0, 0.3) & \text{ if } x\leq 1 \text{ and } y\geq 1,\\
(\rho_3,u_3,v_3,p_3)=(0.138, 1.206, 1.206, 0.029)&\text{ if } x\leq 1\text{ and }y\leq 1,\\
(\rho_4,u_4,v_4,p_4)=(0.5323, 0, 1.206, 0.3) &\text{ if } x\leq1\text{ and } y\leq 1.
\end{array}\right .
$$
Here, the four states are separated by shocks. The domain is $[-2,2]^2$. The solution at $t_f=3$ is displayed in Figure \ref{fig:KT}.
 The mesh is $100\times 100$ cells, i.e. with $120 312$ DoFs. The scheme is that of \cite{BP_Pampa_VEM} where we have applied the bound preserving procedure on the point values only. The CFL is set to $0.1$. The spatial approximation uses the Virtual Finite Element formulation.
 \begin{figure}[hp]
 \begin{center}
 \includegraphics[width=0.8\textwidth]{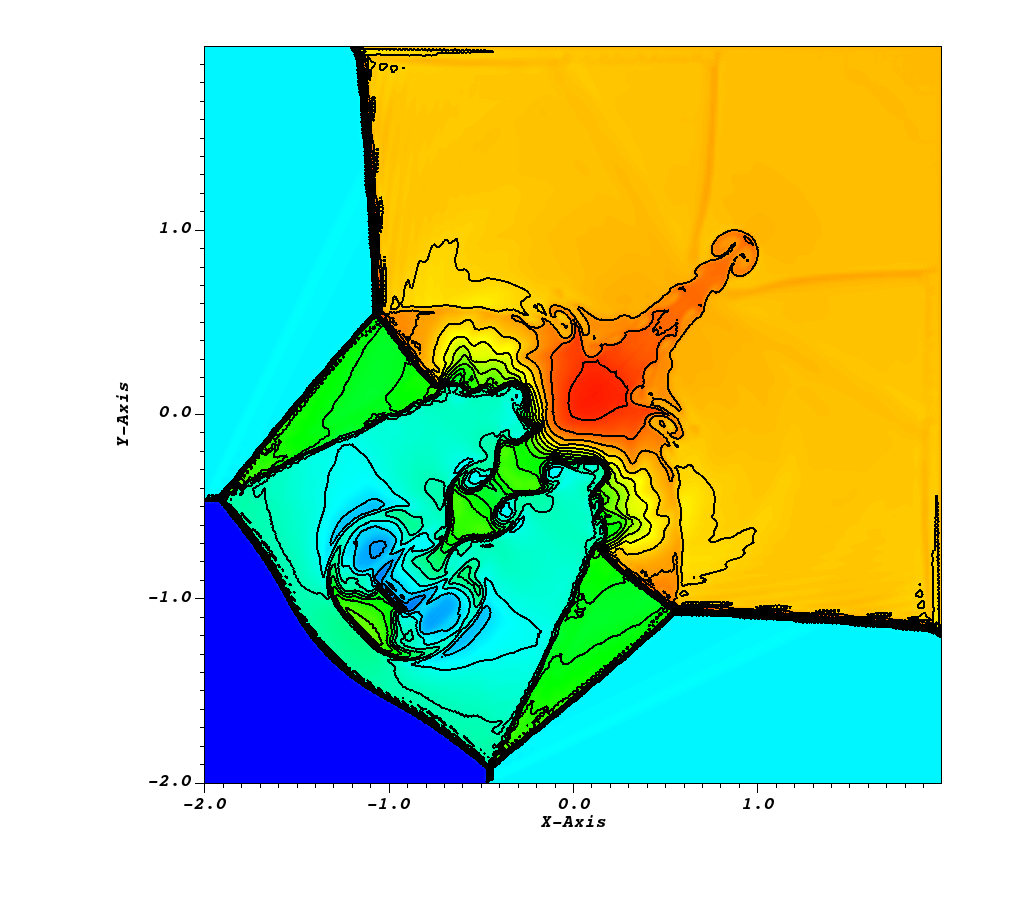}
 \end{center}
 \caption{\label{fig:KT} Solution of Lax-Liu case \# 3.}
 \end{figure}
This is in perfect agreement with proposition \ref{prop}.

\section{Comment on the connection to summation-by-parts}\label{SBP_property}
After our  re-interpretation of the \pampa scheme in terms of \dg, 
we are interested on the underlying structures/properties of the   \pampa scheme even if we have shown that the scheme is energy stable. One key feature of numerical methods is the summation-by-parts (SBP) property. These schemes mimic integration-by-parts discretely and give a general framework to construct energy-stable schemes \cite{Fernandez2014,nordstroem2014,abgrall2020}. Several extensions exists of SBP operators, in addition, in  \cite{Upwind_SBP_operators}, Mattesson introduced upwind SBP operators  which have also been applied inside a \dg framework together with flux vector splitting more recently in \cite{glaubitz2025}. 
The key ingredients of SBP or upwind SBP are the following: 
\begin{itemize}
\item The derivative matrix $D$  which approximates the first derivative of a function and has to be exact up to a certain degree $p$ of the underlying function space, normally polynomial approximation are used however it is not restricted to it, cf. \cite{glaubitz2023FSBP}.
In terms of upwind SBP operators, we have two operators $D_+$ and $D_-$ which have to be exact up to a certain degree $p$ of the polynomial vector space. They are constructed that  $D=(D_++D_-)/2$ holds.
\item In classical SBP and upwind SBP framework, a mass/norm matrix $M$ which mimics the $L^2$ scalar product and often is set be a diagonal matrix where the entries on the diagonal are the quadrature weights. The exactness of the quadrature has to be at least $2p-1$ for the construction process \cite{Fernandez2014SBP} if a polynomial approximation space is considered.
\item The almost skew-symmetric matrix Q:=MD which fulfills than the SBP property meaning 
\begin{equation}\label{eq_SBP}
Q^T+Q=B= \diag{(-1,0,\dots,0, 1)}
\end{equation}
for inflow-outflow boundary conditions.
For periodic boundary conditions, we have instead 
\begin{equation}\label{eq_SBP_2}
Q^T+Q=0.
\end{equation}
In terms of the upwind SBP operators, we have $Q_+$ and $Q_-$ such that $Q_++Q^T_-=0$ and $Q_++Q_+^T = S$, where $S$ is a dissipation matrix. $S$ is symmetric and negative semi-definite. In the upwind case, we have therefore the following conditions: 
\begin{equation}
    \begin{split}
D_+= M^{-1}(Q_{+}  + B/2) & \text{ and } D_-= M^{-1}(Q_- + B/2),\\
MD_+ + D^{T}_-M^T=B & \text{ and } D_+= D_{-} +M^{-1}S.
\end{split}
\end{equation}
\end{itemize}

Coming back to the \pampa \eqref{eq:dG_form2} approach, we can directly define the almost skew-symmetric matrix $Q$ from the SBP community  using \eqref{eq:flux_dG} in each element. It is given by 
$$Q = \begin{pmatrix}
-\frac{1}{2} & 1 & -\frac{1}{2} \\
-1 & 0 & 1 \\
\frac{1}{2} & -1 & \frac{1}{2}
\end{pmatrix}
$$
where $\mathbf{u}= (u_j,\bar{u}_{j+1/2},u_{j+1})^T$.
By direct calculation, \eqref{eq_SBP} is naturally fulfilled with  $B=\operatorname{diag}(-1, 0, 1)$ where we have exactness for constants in the underlying function space due to 
$$D \textbf{1} =(\mathtt{M}^{-1}Q)\textbf{1}= \begin{pmatrix}
-4 & 6 & -2 \\
-1 & 0 & 1 \\
2 & -6 & 4
\end{pmatrix} 
\begin{pmatrix}
1 \\
1\\
1
\end{pmatrix}
=\textbf{0}.
$$
This is essential for local conservation \cite{glaubitz2023FSBP}. The \pampa scheme fulfills essential all the properties which has been proven for SBP schemes or techniques which can be applied inside one element. However, up to this point this is now only the simple \dg formulation with a non-classical approximation space. As it is described in previous chapters, the \pampa scheme is not like this, but use continuous approximations instead of the \dg version. Therefore, we have now to project back from the discontinuous vector space to the continuous one. In terms of  \pampa, we can do this via upwinding or central (or combination of them) at the point values at the element interfaces, cf. Section \ref{se_interpretation}. It means for example when upwinding or a central projection is considered. We have the following cases:
\begin{itemize}
\item \textbf{Upwind}: Drop the update of $u_j$: 
  \begin{itemize}
      \item if $a>0$: 
  coming from the interval $I_\jph$ and keep the internal DoFs computed from \dg. 
  \item If $a<0$, coming from the interval $I_\jmh$ and keep the internal DoFs computed from \dg.
  \end{itemize}
  \item \textbf{Central}: Take the average of the left and right contribution at the point values. 
  \end{itemize}
  Therefore, we have to adapt the point value update in each element concerning the specific framework used where the average value $\xbar{u}_{\jph}$ is not touched. Please remember that our vector inside the element  $I_{\jph}$ for a third order method looks like this $(u_j, \xbar{u}_{\jph}, u_{j+1})^T$ whereas inside $I_{j+\frac{3}{2}}$ we have
  $(u_{j+1}, \xbar{u}_{j+\frac{3}{2}}, u_{j+2})^T$. Focusing on the common element interface at $x_{j+1}$, we have a value for $u_{j+1}^-$ from the left calculated by the DoFs inside $I_\jph$ and alternatively one value from the right $u_{j+1}^{+}$.
  Using \eqref{eq:dG_form2}, we have for the update of $u_{j+1}^{-/+}$ the following 
\begin{equation}\label{eq:dG_form2_2}
  \frac{\mathrm d}{\mathrm dt}\begin{pmatrix}
                                 u_j \\
                               \xbar{u}_\jph \\
                                 \PO{u_{j+1}^-} \\
                                 \PO{u_{j+1}^+}\\
                                \xbar{u}_{j+\frac{3}{2}}\\
                                u_{j+2}\\
                               \end{pmatrix}
  +\frac{a}{\dx}\begin{pmatrix}
      6\xbar{u}_\jph-4u_j-2u_{j+1}^- \\
      u_{j+1}^--u_j \\
      2u_j+4u_{j+1}^--6\xbar{u}_\jph \\
        6\xbar{u}_{j+\frac{3}{2}}-4u_{j+1}^+-2u_{j+1}^+ \\
      u_{j+2}-u_{j+1}^+\\
      2u_{j+1}^++4u_{j+2}-6\xbar{u}_\jph 
    \end{pmatrix}=0.
\end{equation}
This would be the \dg update where we focus on the two neighboring elements $I_\jph$ and $I_{j+\frac{3}{2}}$. In this context, the  update of $u_{j+1}$ would be discontinuous over the element boundaries therefore we have to project back to the continuous approximation space. For simplicity reason later on, we select the central scheme. Here, we take the average to define the continuous representation over this cell interface $u_{j+1}= \frac{u_{j+1}^++u_{j+1}^-}{2}$. This yields to the update for $u_{j+1}$:
\begin{equation*}
  \frac{\mathrm d}{\mathrm dt}
                                 \PO{u_{j+1}}
  +\frac{a}{\dx}
    \frac{1}{2} \left( 2u_j+4u_{j+1}-6\xbar{u}_\jph + 6\xbar{u}_{j+\frac{3}{2}}-4u_{j+1}-2u_{j+2}  \right) =0.
\end{equation*}
resulting in 
\begin{equation*}
  \frac{\mathrm d}{\mathrm dt}
                                 \PO{u_{j+1}}
  +\frac{a}{\dx}
 \left( u_j-3\xbar{u}_\jph + 3\xbar{u}_{j+\frac{3}{2}}-u_{j+2}  \right) =0.
\end{equation*}
 This projection operator works on all interface values and therefore, 
 we have instead of three updates in each cell, in \pampa we have only two  to be considered. The updated derivative matrix has a banded structure and is given by
 $$
 \tilde{D}=\frac{1}{\Delta x} \begin{pmatrix}
\ddots &  &  &  &  &  &  \\
-1 & 0 & 1 &  &  &  &  \\
1 & -3 & 0 & 3 & -1 &  &  \\
\hline
 & & -1 & 0 & 1 &  &   \\
 &  & 1 & -3 & 0 & 3 & -1 \\
 \hline
 & &  &  & -1 & 0 & 1 \\
 &  &  &  &  &  & \ddots 
 \end{pmatrix},
 $$
 where the embedded values represent one update inside an element. Here, we can easily derive a diagonal mass matrix via $\mathbf{\tilde{M}}= \frac{\Delta x}{4} \diag{(\cdots,|3,1|,3,1,3,1, \cdots)}$. Again, the embedded version corresponds to the contribution in one element. Together we obtain 
 $\mathbf{\tilde{M}} \tilde{D}+ (\mathbf{\tilde{M}} \tilde{D})^T=0 $ which corresponds to an periodic SBP operator. Therefore, we can apply all the results of SBP in such context and obtain stability.
 \begin{rmk}
Instead of using the central scheme, also upwind or other combination can be used. This would result in several different version of upwind SBP operators. Thanks to our proposition \eqref{pro_energy_stability}, energy stability can be proven in such context. However, the interesting question rises how the projection step from the discontinuous approximation space back to the continuous one effects the SBP properties. This is not only interesting in one-space dimension but in particular relevant in multi-dimension. Some works  already exist concerning projection and extension of stability properties inside the SBP framework \cite{OlssonSBP, OlssonSBP2}. Future work will be therefore considering the effect of the projection step on the SBP properties and the results emerged from this.
 \end{rmk}
\begin{rmk}
During the ICOSAHOM 2025 conference in Montr\'eal, a presentation was given on an ongoing work titled "Stability of the Active Flux Method in the Framework of Summation-by-Parts Operators" \cite{barsukow2025stabilityactivefluxmethod}. While we have taken note of this contribution, we would like to clarify that our work was developed independently. There has been no exchange of ideas with this group, which began prior to their presentation.
Our investigation originates from a visit of one of the authors to SUSTech in early June 2025, during which valuable discussions were held with Professor Kailiang Wu, whose input we gratefully acknowledged.
While both efforts concern the active flux method, the focus and context differ: our study is situated more within the finite element framework, in contrast to the finite difference setting emphasized in \cite{barsukow2025stabilityactivefluxmethod}. In particular, their work offers a contribution through a detailed analysis of mass matrix evaluation in the upwind active flux method. In contrast, our interest lies more in understanding the role and properties of the projection step from discontinuous to continuous approximation spaces, which we intend to explore further in future work. In our study, we did not pursue a in-depth stability analysis of the upwind SBP formulation, as Proposition 3.1 already ensures the relevant stability properties, rendering a detailed investigation unnecessary.
\end{rmk}

{
\section{Conclusions, perspectives.}
In this paper, we have shown a connection between the \pampa scheme and the \dg formulation. One can see \pampa, for linear advection problem, as one step of \dg followed by one projection onto the initial approximation step. There are several possible ways to project, we have shown two, and mainly more are certainly possible: what is the "best" one, in which sense?  This reformulation of \pampa leads to a generalisation to non linear problems. This generalisation is not equivalent to the formulation contained in \cite{Abgrall2023b} and \cite{BP_Pampa_VEM} for non linear problem. In \cite{arXiv:2511.16180} this reformulation of \pampa has been used to \remi{formulate} in a much better way the discretisation of boundary conditions. We have also \remi{shown }some intrinsic bound preserving properties of \pampa, linked to a generalisation of the Simpson formula. They have already been used in other publications to construct provable bound preserving schemes, and we also note that other solutions to the same problem are also available in the literature. Last, in one dimension, we have shown a SBP property of \pampa. One can imagine that a generalisation of this to multiD would provide an answer to the first question above: what is the "best" projection.
}
  \section*{Acknowledgments.}
  We thank the many discussions we have had with Professor Kailiang Wu in SUSTech, China, in early June 2025.  Y.L. was supported by UZH Postdoc Grant, 2024 / Verf\"{u}gung Nr. FK-24-110 and SNFS grant 200020$\_$204917 ``Solving advection dominated problems with high order schemes with polygonal meshes: application to compressible and incompressible flow problems''. P. \"O. is supported by the DFG within SPP 2410, project
525866748 and under the personal grant 520756621.
\bibliographystyle{unsrt}
\bibliography{main}
\appendix

\section{VEM approximation: basic facts.}\label{appendix:VEM}
 
Any basis of $V_k(P)$ is virtual, meaning that it is not explicitly computed in closed form. Consequently, the evaluation of $v_h(\bbx)$ at some $\bbx\in P$ it is not straightforward. One way to proceed would be to evaluate $\pi(v_h)$, that is the $L^2$ projection of $v_h$ on $\P^k(P)$. Since we want to do it only using the degrees of freedom, it turns out that this is impossible in practice. But, as shown in \cite{hitch,vem}, it is possible to define a space $W_k(P)$ for which computing the $L^2$ projection is feasible.
It is constructed from $V_k(P)$ in two steps. First, we consider the approximation space $\widetilde{V}_k(P)$ given by
$$\widetilde{V}_k(P)=\{ v_h, v_h \text{ is continuous on }\partial P \text{ and } (v_h)_{\partial P}\in \P^k(\partial P); \text{ and }\Delta v_h\in \P^{k-2}(P)\}.$$
For $p\in \N$, let $M^\star_p(P)$ be the vector space generated by the scaled monomial of degree $p$ exactly,
$$m(\bbx)\in M^\star_p(P), \qquad m(\bbx)=\sum_{\bsa, |\bsa|=p} \beta_{\bsa} m_{\bsa}(\bbx).$$
Then, we consider $W_k(P)$, which is the subspace of $\widetilde{V}_k(P)$ defined by 
$$W_k(P)=\{ w_h\in \widetilde{V}_k(P), \langle w_h-\pi^\nabla w_h,q\rangle=0, \quad \forall q\in M^\star_{k-2}(P)\cup M^\star_k(P)\}.$$
In \cite{vem}, it is shown that $\dim V_k(P)=\dim W_k(P)$.

This approximation space is defined as follows. $w_h\in W_k(P)$ if and only if
 \begin{enumerate}
  \item $w_h$ is a polynomial of degree $k$ on each edge $e$ of $P$, that is $(w_h)_{|e}\in \P^k(e)$,
 \item $w_h$ is continuous on $\partial P$,
\item $\Delta w_h\in \P^{k-2}(P)$,
\item $\int_{P} w_h m_{\bsa} \; {\rm d}\bbx=\int_{P} \pi^\nabla w_h m_{\bsa} \; {\rm d}\bbx$ for $|\bsa|=k-1, k$.
\end{enumerate}
The degrees of freedom are the same as in $V_k(P)$:
\begin{enumerate}
 \item The value of $w_h$ on the vertices of $P$,
 \item On each edge of $P$, the value of $v_k$ at the $k-1$ internal points of the $k+1$ Gauss--Lobatto points on this edge,
 \item The moments up to order $k-2$ of $w_h$ in $P$,
 $$m_{\bsa}(w_h):=\dfrac{1}{\vert P\vert }\int_{P} w_h m_{\bsa}\; d\bbx, \qquad |\bsa|\leq k-2.$$
 \end{enumerate}

The $L^2$ projection of $w_h$ is computable. For any $\bsb$, if $\pi^0(w_h)=\sum_{\bsa, |\bsa|\leq k} s_{\bsa} m_{\bsa}$, we have
$$\langle \pi^0(w_h), m_{\bsb}\rangle=\sum_{\bsa, |\bsa|\leq k} s_{\bsa} \langle m_{\bsb}, m_{\bsa}\rangle=\langle w_h, m_{\bsb}\rangle.$$
The left hand side is computable since the inner product $\langle m_{\bsb}, m_{\bsa}\rangle$ only involves monomials. We need to look at the right hand side. If $|\bsa|\leq k-2$, $\langle w_h, m_{\bsb}\rangle=|P| m_{\bsb}(w_h)$ and if $|\bsa|=k-1$ or $k$, we have
$$\langle w_h, m_{\bsb}\rangle=\langle \pi^\nabla w_h, m_{\bsa}\rangle,$$ which is computable from the degrees of freedom only.

We note that if $w_h\in W_k(P)$, then $m_{\bsa}(\pi^0 w_h)=m_{\bsa}(w_h)$. Indeed
$$m_{\bsa}(\pi^0 w_h)=\frac{1}{|P|} \langle \pi^0 w_h ,m_{\bsa} \rangle =\frac{1}{|P|} \langle  w_h ,m_{\bsa} \rangle$$ by construction. This is not true for $\pi^\nabla$ in $V_k(P)$.
However, for $k\leq 2$, $V_k(P)=W_k(P)$

The last remark is that, since $\P^k(P)\subset V_k(P)$ and $\P^k(P)\subset W_k(P)$, the projections of $C^{k+1}(P)$ onto $V_k(P)$ and $W_k(P)$ defined by the degrees of freedom is $k+1$-th order accurate.

\bigskip

We recall how to approximate the mass matrix.
We have for any $i$ and $j$ that represent the index of any degree of freedom,
$$\int_P\varphi_i\varphi_j=\big (\pi \varphi_i,\pi\varphi_j\big )+
\big (\Id-\pi)\varphi_i, (\Id-\pi)\varphi_j)$$ because for any $u\in L^2(K)$, we have,v
$$\big ( (\Id-\pi)u,\pi v\big )=0$$
The term
$$\big (\pi \varphi_i,\pi\varphi_j\big )$$ can be computed exactly, while
we approximated
$$\big ( \big (\Id-\pi)\varphi_i, (\Id-\pi)\varphi_j)\approx \vert P\vert \sum_{r=1}^{N_{dofs}} DOF_r\big ( (\Id-\pi)\varphi_i\big ) 
DOF_r\big ( (\Id-\pi)\varphi_j\big )$$
The sum of the two matrices is invertible.

\end{document}